\documentclass[11pt,a4paper,reqno]{amsart}
\usepackage{color}
\usepackage{amsfonts,amsmath,amssymb,amsxtra,url,float} % 论文初始格式设定
\usepackage[colorlinks,linkcolor=RoyalBlue,anchorcolor=Periwinkle,citecolor=Orange,urlcolor=Green]{hyperref} % 超链颜色设定
\usepackage[usenames,dvipsnames]{xcolor} % 颜色宏包调用
\usepackage{enumitem}
\usepackage{geometry} 
\usepackage{graphicx} % 图片宏包调用, 调用此宏包后将被允许插入png 格式图片
\usepackage{subfigure} % figure环境下的补充宏包, 允许在figure环境下对自图片进行分栏
\usepackage{tikz}\usetikzlibrary{matrix}\usetikzlibrary{trees}
\usetikzlibrary{matrix}
\usetikzlibrary{patterns}
\usetikzlibrary{shadings}\usepgflibrary{shadings} % 调用tikz作图指令包
\usepackage[all]{xy} % 交换图箭图宏包调用, 允许在xymatrix环境下进行交换图表编译
\usepackage{mathdots} % 省略记号宏包调用, 调用此宏包后将被允许使用竖直、倾斜等多种省略符号
\usepackage{yhmath} %
\usepackage{dsfont} % 可以调用mathds指令
\usepackage{cite}
\usepackage{mathrsfs} % mathscr
\usepackage{multicol} %用于实现在同一页中实现不同的分栏
\usepackage{changepage}
\numberwithin{figure}{section}

\usepackage{marginnote} % 可使用边注\marginnote{}指令
\usepackage{graphicx} % 可使用\figure环境
\usepackage{multicol} %用于实现在同一页中实现不同的分栏
\newtheorem{theorem}{Theorem}[section]
\newtheorem{lemma}[theorem]{Lemma}
\newtheorem{corollary}[theorem]{Corollary}
\newtheorem{main theorem}[theorem]{Main Theorem}
\newtheorem{proposition}[theorem]{Proposition}
\newtheorem{definition}[theorem]{Definition}
\newtheorem{construction}[theorem]{Construction}
\newtheorem{question}[theorem]{Question}
\newtheorem{remark}[theorem]{Remark}
\newtheorem{example}[theorem]{Example}
\newtheorem{notation}[theorem]{Notation}

\usetikzlibrary{arrows}

\numberwithin{equation}{section}

\def\<{\langle} % 左尖括号指令省略为"\<"
\def\>{\rangle} % 右尖括号指令省略为"\>"
\def\NN{\mathbb{N}} % mathkk双写字体的"N", 自然数集
\def\ZZ{\mathbb{Z}} % mathkk双写字体的"Z", 整数集
\newcommand{\Pic}{F{\tiny{IGURE}}\ }
\newcommand{\modcat}{\mathsf{mod}}

\newcommand{\ind}{\mathsf{ind}}
\newcommand{\kk}{\mathds{k}} % mathds双写字体的"k", 用于描述域
\newcommand{\Q}{\mathcal{Q}} % mathcal花体的"Q", 用于描述箭图
\newcommand{\I}{\mathcal{I}} % mathcal花体的"I", 用于描述admissible 理想
\newcommand{\per}{\mathsf{per}} % perfect cat.
\newcommand{\Hom}{\mathrm{Hom}} %
\newcommand{\End}{\mathrm{End}} %
\newcommand{\grEnd}{\mathcal{E}nd} %
\renewcommand{\H}{\mathrm{H}} %
\newcommand{\SURF}{\mathbf{S}} % 用于描述标记曲面的记号, 粗体的"S"
\newcommand{\Surf}{\mathcal{S}} % 用于描述曲面的记号, mathcal花体的"S"
\newcommand{\bSurf}{\partial\mathcal{S}} % 用于描述曲面边界的记号
\newcommand{\M}{\mathcal{M}} % 用于描述标记点的记号, mathcal花体的"M"
\newcommand{\MM}{\mathfrak{M}} % M([c]) is the indecomposable module corresponded by c.
\newcommand{\gbullet}{{\color{ForestGreen}\bullet}} % 用于描述open 标记点的记号
\newcommand{\rbullet}{{\color{red}\circ}} % 用于描述额外标记点或closed 标记点的记号
\newcommand{\E}{\mathcal{E}} % 用于描述额外标记点的记号, mathcal 花体的"E"
\newcommand{\Dgreen}{\Delta_{\color{ForestGreen}\bullet}} % open FFAS
\newcommand{\Dred}{\Delta_{\color{red}\circ}} % colosed FFAS
\newcommand{\tDgreen}{\widetilde{\Delta}_{{\color{ForestGreen}\bullet}}} % open marked points上的分次全形式弧系
\newcommand{\tDred}{\widetilde{\Delta}_{{\color{red}\circ}}} % closed marked points上的分次全形式弧系
\newcommand{\PP}{\mathcal{P}} % elementary polygon
\newcommand{\innerSurf}{\mathcal{S}\backslash\partial\mathcal{S}} % 用于描述曲面的内部
\newcommand{\tc}{\tilde{c}} % graded curves
\newcommand{\Y}{\mathcal{Y}} % the set of all closed marked points
\newcommand{\F}{\mathcal{F}} % grading of marked surface
\newcommand{\ii}{\mathfrak{i}} % intersection index
\newcommand{\PC}{\mathrm{PC}} % the set of all permissible curves
\newcommand{\CC}{\mathrm{CC}} % the set of all permissible closed curves
\newcommand{\AC}{\widetilde{\mathrm{AC}}} % the set of all graded open curves
\newcommand{\A}{\mathbb{A}}
\newcommand{\tA}{\tilde{\mathbb{A}}}
\newcommand{\ta}{\tilde{a}}
\newcommand{\ared}[2]{\tilde{a}^{#1}_{\rbullet, #2}}

\newcommand{\agreen}[2]{a^{#1}_{\gbullet, #2}}

\newcommand{\m}{\mathfrak{m}}
\newcommand{\cemb}{\mathfrak{C}}
\newcommand{\X}{\mathfrak{X}}
\renewcommand{\top}{\mathrm{top}}
\newcommand{\soc}{\mathrm{soc}}
\newcommand{\simp}{\mathrm{simp}}
\newcommand{\proj}{\mathrm{proj}}
\newcommand{\longline}{-\!\!\!-\!\!\!-}

\newcommand{\bfP}{\pmb{P}}
\newcommand{\bfI}{\mathbf{I}}
\newcommand{\bfII}{\mathbf{II}}
\newcommand{\bfIII}{\mathbf{III}}
\newcommand{\Tri}{\mathrm{Tri}}
\newcommand{\stautilt}{\mathrm{s}\tau\text{-}\mathrm{tilt}}
\newcommand{\gldim}{\mathrm{gl.dim}}
\newcommand{\tGamma}{\widetilde{\Gamma}}
\def\defines{\it\color{black}}

\begin{document}
%=========================================================

\title[There are no strictly shod algebras in hereditary gentle algebras]{There are no strictly shod algebras in hereditary gentle algebras}
\thanks{$^{\ast}$Corresponding author.}
\thanks{MSC2020: 05E10, 16G10, 16E45.}
\thanks{Key words: strictly shod algebra, graded marked surface, silted algebra, hereditary gentle algebra}
\author{Houjun Zhang}
\address{Houjun Zhang, School of Science, Nanjing University of Posts and Telecommunications, Nanjing 210023, P. R. China}
\email{zhanghoujun@njupt.edu.cn}
\author{Yu-Zhe Liu$^{\ast}$}
\address{Yu-Zhe Liu, School of Mathematics and statistics, Guizhou University, Guiyang 550025, P. R. China}
\email{yzliu3@163.com}

%\dedicatory{}
%\subjclass[2020]{\color{red} xxOxx; xxOxx; xxOxx... }
%\keywords{\color{red} keywords; keywords; keywords... }
%\thanks{\color{red} This work is supported by ... }

%=========================================================

%\definecolor{section}{rgb}{0,0,0.5}

%\marginpar{\tiny\listofchanges}

\begin{abstract}
We prove that there are no strictly shod algebras in hereditary gentle algebras by geometric models. As an application, we give a classification of the silted algebras for Dynkin type $\mathbb{A}_{n}$ and $\widetilde{\mathbb{A}}_{n}$.
\end{abstract}

\maketitle

%=========================================================
\section{Introduction}

Silted algebras were introduced by Buan and Zhou \cite{BZ2016} as endomorphism algebras of 2-term silting complexes over finite dimensional hereditary algebras. They showed that any silted algebra is shod \cite{CL1999} (=small homological dimension, that is, the projective or the injective dimension is at most one). In particular, for a connected finite dimensional algebra $A$, they proved that $A$ is silted if and only if it is tilted or strictly shod algebras which are the shod algebras of global dimension three. 

Recall that in \cite[Theorem 1.1 (a)]{BZ2018}, Buan and Zhou showed that the global dimension of any silted algebra is less than or equal to three. The global dimension of any tilted algebra is at most two. Thus, we are interested in the strictly shod algebras. Recently, Xing gave some examples of silted algebras for the path algebras of some Dynkin quivers. She obtained that there are strictly shod algebras for the path algebras of Dynkin type $\mathbb{D}_{5}$. For example: 

\begin{example}\rm Let $Q$ be the following quiver
$$
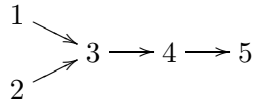
\begin{xy}
(-10,5)*+{1}="0",
(-10,-5)*+{2}="1",
(0,0)*+{3}="2",
(10,0)*+{4}="3",
(20,0)*+{5}="4",
\ar"0";"2",\ar"1";"2", \ar"2";"3", \ar"3";"4",
\end{xy}$$ and $D_{5}=\kk Q$. 
We take $S=P(5)\oplus P(4)\oplus P(2)\oplus I(2)\oplus P(2)[1]$, so $S$ is a 2-term silting complex over $D_{5}$. Then the quiver of $\End{(S)}$ is 
$$\begin{xy}
(20,0)*+{\circ}="5",
(10,0)*+{\circ}="4",
(0,0)*+{\circ}="3",
(-10,0)*+{\circ}="2",
(-20,0)*+{\circ}="1",
\ar"1";"2", \ar"2";"3", \ar"3";"4",\ar"4";"5",\ar@/^0.8pc/@{.}"1";"3",\ar@/^0.8pc/@{.}"2";"4",
\end{xy}.
$$ Thus the global dimension of $\End{(S)}$ is 3, that is it is a strictly shod algebra.

\end{example}

However, she does not found strictly shod algebras for the path algebras of Dynkin type $\mathbb{A}_{3,4}$. Inspired by the work of Xing, in \cite{XYZ2022}, Xie, Yang and the first author of this paper gave a complete classification of the silted algebras for the quiver $\overrightarrow{\mathbb{A}}_{n}$ of Dynkin type $\mathbb{A}_{n}$ with linearly orientation and the quiver obtained by reversing the arrow of quiver $\overrightarrow{\mathbb{A}}_{n}$ at the unique sink. We also obtained that there are no strictly shod algebras for this two Dynkin type $\mathbb{A}_{n}$. Thus, a natural question is:

\begin{question} \label{qus:question}
 Whether there are no strictly shod algebras in Dynkin type $\mathbb{A}_{n}$ with arbitrary orientation ?
\end{question} 

In this paper, we mainly study this question and give a positive answer. Note that any path algebra of Dynkin type $\mathbb{A}_{n}$ is a gentle algebra, which were introduced  by Assem and Skowro\'{n}ski \cite{AS1987} in 1980. Recently, the geometric models of gentle algebras have become an important research object and have been studied in representation theory by a number of authors, for example \cite{BCS2019, HKK2017, OPS2018,APS2019}.
In \cite{LZ2022}, we gave a geometric characterization of the silted algebras for the path algebra of Dynkin type $\mathbb{A}_{n}$ with linearly orientation. As application, we also gave a classification of the silted algebras for the path algebra of Dynkin type $\mathbb{A}_{n}$ with linearly orientation. However, this approach does not able to solve Question \ref{qus:question}. In order to answer Question \ref{qus:question}, we introduce 
a curve embedding from the geometric models of hereditary gentle algebras for module categories to derived categories. Notice that the hereditary gentle algebras in this paper are the path algebras of quivers of type $\mathbb{A}_{n}$ and $\tA_{n}$.
Let $A$ be a hereditary gentle algebra and $\SURF^{\F_A}(A)$ be the marked ribbon surface of $A$. Then, we have

\begin{theorem}[Theorem \ref{thm:curve embeding}] \label{thm1}
There is an injection
\[\cemb: \PC(\SURF^{\F_A}(A)) \to \AC(\SURF^{\F_A}(A)),\ c\mapsto \tc^{\circlearrowleft}\]
from the set of all equivalent classes of permissible curves in $\SURF^{\F_A}(A)$
to the set of all equivalent classes of admissible curves such that the indecomposable complex corresponding to $\tc^{\circlearrowleft}$ is quasi-isomorphic to the minimal projective representation of any string module corresponding to $c$.
%$($The definition of $\tc^{\circlearrowleft}$ is given in the subsection \ref{subsect:curves corresp. string comp.}.$)$
\end{theorem}

It is known that a full formal arc system of the marked surface of $A$ divides the marked surface of $A$ into some elementary polygons such that there is exactly one edge in the boundary of the marked surface of $A$. In \cite[Theorem 5.10]{LGH2022}, the second author of this paper and his coauthors proved that the global dimension of gentle algebras can be calculated by the edges of the elementary polygons. Then by Theorem \ref{thm1} and the characterization of the silted algebras of gentle algebras in \cite{LZ2022}, we have

\begin{theorem}[Theorem \ref{main thm}] \label{thm2}
There are no strictly shod algebras in hereditary gentle algebras.
\end{theorem}

As a consequence, by Theorem \ref{thm2}, we have the following two corollaries.

\begin{corollary}
Let $A$ be the path algebra of Dynkin type $\mathbb{A}_{n}$ with arbitrary orientation. Then the silted algebras of $A$ have two forms:
\begin{itemize}
\item[{\rm(1)}] the tilted algebras of type $\mathbb{A}_{n}$;

\item[{\rm(2)}] the direct product of some tilted algebras of type $\mathbb{A}_{m_{1}},\ldots,\mathbb{A}_{{m}_{k}}$ such that $m_{1}+\cdots+m_{k}=n$ $(k\ge 2)$.
\end{itemize}
\end{corollary}

\begin{corollary}
Let $A$ be the path algebra of type $\tA_{n}$ with arbitrary orientation. Then the silted algebras of $A$ have four forms:
\begin{itemize}
\item[{\rm(1)}] the tilted algebras of type $\tA_n$;

\item[{\rm(2)}] the tilted algebras of type $\A_n$;

\item[{\rm(3)}] the direct product of some tilted algebras of type $\A_{m_1},\ldots,\A_{m_k}$ such that $m_{1}+\cdots+m_{k}=n$ $(k\ge 2)$;

\item[{\rm(4)}] the direct product of some tilted algebras of type $\A_{m_{1}},\ldots,\A_{{m}_{k}}, \tA_{\tilde{m}}$ such that $m_{1}+\cdots+m_{k} + \tilde{m}=n$ $(k\ge 1)$.
\end{itemize}
\end{corollary}

This paper is organized as follows. In Section 2 we recall some preliminaries on the geometric models of gentle algebras. In Section 3 we study the curve embedding from the geometric models of gentle algebras for module categories to derived categories. In the last section, we prove Theorem \ref{thm2}.

Throughout this paper, we always assume that $\kk$ is an algebraically closed field and any surface $\Surf$ is smooth and its boundary $\bSurf$ is non-empty.

\section{Hereditary gentle algebras and their geometric models}

In this section, we recall some notations and concepts for geometric models of hereditary gentle algebras, we refer the reader to  \cite{HKK2017, OPS2018, BCS2019, APS2019, AS1987} for more details.

\subsection{Graded marked ribbon surfaces} \label{subsect:geo.mod.}
A {\defines marked surface} is a triple $(\Surf, \M, \Y)$, where $\M$ and $\Y$ are finite subsets of the boundary $\bSurf$ such that elements in $\M$ and $\Y$ are alternative in every boundary component. Elements in $\M$ and $\Y$ are called {\defines open marked points} (=$\gbullet$-marked points) and {\defines closed marked points} (=$\rbullet$-marked points), respectively.

An {\defines open curve} (=$\gbullet$-curve) (resp., a {\defines closed curve} (=$\rbullet$-curve))
is a curve in $\Surf$ whose endpoints are $\gbullet$-marked points (resp., $\rbullet$-marked points).
Especially, we always assume that any two curves in $\Surf$ are representatives in their homotopy classes such that their intersections are minimal.

A {\defines full formal open arc system} (=$\gbullet$-FFAS) (resp., {\defines full formal closed arc system} (=$\rbullet$-FFAS)) of $(\Surf, \M, \Y)$, say $\Dgreen$ (resp., $\Dred$),
is a set of some $\gbullet$-curves (resp., $\rbullet$-curves) such that:
\begin{itemize}
  \item any two $\gbullet$-curves (resp., $\rbullet$-curves) in $\Dgreen$ (resp., $\Dred$) has no intersection in $\innerSurf$;
  \item every {\defines elementary $\gbullet$-polygon} (resp., {\defines elementary $\rbullet$-polygon}), the polygon obtained by $\Dgreen$ (resp., $\Dred$) cutting $\Surf$, has a unique edge in a subset of $\bSurf$.
\end{itemize}
The elements of $\Dgreen$ (resp., $\Dred$) are called {\defines $\gbullet$-arcs} (resp., {\defines $\rbullet$-arcs}).

Next we introduce the definition of the marked ribbon surfaces of hereditary gentle algebras.

\begin{definition} \rm
A {\defines marked ribbon surface} $\SURF=(\Surf, \M, \Y, \Dgreen, \Dred)$ is a marked surface $(\Surf, \M, \Y)$ with $\gbullet$-FFAS $\Dgreen$ and $\rbullet$-FFAS $\Dred$ such that the following conditions hold.
\begin{itemize}
  \item $\Dgreen$ is a {\defines triangulation} of $\Surf$, i.e.,
    any elementary $\gbullet$-polygon is a digon or triangle.
  \item $\Dred$ is the {\defines dual dissection} of $\SURF$ decided by $\Dgreen$,
  i.e., for every $\rbullet$-curve $a_{\rbullet}$ (resp., $\gbullet$-curve $a_{\gbullet}$), there is a unique $\gbullet$-curve $a_{\gbullet}$ (resp., $\rbullet$-curve $a_{\rbullet}$) intersecting with $a_{\rbullet}$ (resp., $a_{\gbullet}$)
  and $a_{\gbullet}$ and $a_{\rbullet}$ has only one intersection.
\end{itemize}

A {\defines graded surface} $\Surf^{\F}=(\Surf, \F)$ is a surface $\Surf$ with a section $\F$ of the projectivized tangent bundle $\mathbb{P}(T\Surf)$. $\F$ is called a {\defines foliation} (or {\defines grading}) on $\Surf$.

A {\defines graded curve} in a graded surface $(\Surf, \F)$ is a pair $(c, \tc)$ of curve $c: [0,1] \to \Surf$ and {\defines grading} $\tc$,
where $\tc$ is a homotopy class of paths in $\mathbb{P}(T_{c(t)}\Surf)$ from the subspace given by $\F$ to the tangent space of the curve, varying continuously with $t\in [0,1]$.
For simplicity, we denoted by $c$ and $\tc$ the curve $c:[0,1]\to \Surf$ and the graded curve $(c, \tc)$, respectively.
Let $\tc_1$ and $\tc_2$ be two graded curves and $p$ be an intersection of $c_1$ and $c_2$.
The {\defines intersection index} $\ii_p(\tc_1, \tc_2)$ of $p$ is an integer given by
\[\ii_p(\tc_1, \tc_2) := \tc_1(t_1) \cdot \kappa_{12} \cdot \tc_2(t_2)^{-1} \in \pi_1(\mathbb{P}(T_p\Surf)) \cong \mathbb{Z}, \]
where
\begin{itemize}
  \item $\tc_i(t_i)$ is the homotopy classes of paths from $\F(p)$ to the tangent space $\dot{c}_i(t_i)$;
  \item $\kappa_{12}$ is the paths from $\dot{c}_1(t_1)$ to $\dot{c}_2(t_2)$ given by clockwise rotation in $T_p\Surf$ by an angle $<\pi$;
  \item $\pi_1(\mathbb{P}(T_p\Surf))$ is the fundamental group defined on the projectivized of the tangent space $T_p\Surf$.
\end{itemize}

A {\defines graded full formal open arc system} (=$\gbullet$-grFFAS) (resp., {\defines graded full formal closed arc system} (=$\rbullet$-grFFAS)) is an $\gbullet$-FFAS (resp., $\rbullet$-FFAS) whose elements are graded $\gbullet$-curves (resp., $\rbullet$-curves).

A {\defines graded marked ribbon surface} $\SURF^{\F}=(\Surf^{\F}, \M, \Y, \tDgreen, \tDred)$ is a marked ribbon surface with a foliation $\F$
such that $\gbullet$-FFAS and $\rbullet$-FFAS are graded and $\ii_{a_{\gbullet}\cap a_{\rbullet}}(\ta_{\gbullet}, \ta_{\rbullet})=0$.
\end{definition}

\begin{notation} \rm
We say that $a_{\gbullet}$ (resp., $a_{\rbullet}$) is the {\defines dual arc} of $a_{\rbullet}$ (resp., $a_{\gbullet}$)
and all $\rbullet$-marked points lying in digon are {\defines extra marked points}.
Denote by $\E$ the set of all extra marked points of $\SURF^{\F}$.
%and by $\OEP(\SURF^{\F})$ (resp., $\CEP(\SURF^{\F})$) the set of all elementary $\gbullet$-polygons (resp., elementary $\rbullet$-polygons).
\end{notation}

Any graded marked ribbon surface $\SURF^{\F}$ induces a graded algebra as follows.

\begin{construction} \label{construction} \rm
The graded algebra $A(\SURF^{\F})$ of $\SURF^{\F}$ is a finite dimensional algebra $\kk\Q$ with grading $|\cdot|: \Q \to \ZZ$ given by the following steps.
\begin{itemize}
  \item[Step 1]
    there is a bijection $\mathfrak{v}: \Dgreen \to \Q_0$, i.e., $\Q_0 = \Dgreen$;
  \item[Step 2]
    any elementary $\gbullet$-polygon $\PP$ given by $\Dgreen$ provides an arrow $\alpha: \mathfrak{v}(a^1_{\gbullet}) \to \mathfrak{v}(a^2_{\gbullet})$,
    where $a^1_{\gbullet}, a^2_{\gbullet}\in \Dgreen$ are two edges of $\PP$ with common endpoints $p\in \M$
    and $a^2_{\gbullet}$ is left to $a^1_{\gbullet}$ at the point $p$;
  \item[Step 3]
    the grading $|\alpha|$ of $\alpha$ equals to the intersection index $1-\ii_{a^1_{\rbullet}\cap a^2_{\rbullet}}(\ta^1_{\rbullet}, \ta^2_{\rbullet})$.
\end{itemize}
\end{construction}

In order to describe the indecomposable modules of gentle algebras, Baur and Coelho-Sim\~{o}es \cite{BCS2019} defined permissive curves as follows.

\begin{definition}\rm
A {\defines permissible curve} $c$ is a function $c: [0,1] \to \Surf$ such that the following conditions hold.
\begin{itemize}
    \item[(1)] The endpoints of $c$ are elements belong to $\M\cup\E$, or $c(0)=c(1)\in\innerSurf$;
    \item[(2)] If $c$ {\defines consecutively crosses} two $\gbullet$-arcs $a_{\gbullet}^1$ and $a_{\gbullet}^2$,
     that is, the segment $c_{1,2}$ of $c$ obtained by $a_{\gbullet}^1$ and $a_{\gbullet}^2$ cutting $c$ dose not cross other $\gbullet$-arcs, then $a_{\gbullet}^1$ and $a_{\gbullet}^2$ have a common endpoint lying in $\M$.
\end{itemize}
\end{definition}

Furthermore, Opper, Plamondon and Schroll \cite{OPS2018} used the admissible curves to characterize the indecomposable objects in derived categories.

\begin{definition}\rm
An {\defines admissible curve} $\tc$ is a graded $\gbullet$-curve or a graded curve without endpoints.
\end{definition}

\subsection{Hereditary gentle algebras} \label{subsec:hga}

\begin{definition}\rm
Let $\Q$ be a finite quiver and $I$ an admissible ideal of the path algebra $\kk\Q$. Then algebra $A=\kk\Q/I$ is called a {\it gentle algebra} provided the following conditions are satisfied:
\begin{itemize}
  \item[(G1)] Each vertex of $\Q$ is the source of at most two arrows and the target of at most two arrows.

  \item[(G2)] For each arrow $\alpha:x\to y$ in $\Q$, there is at most one arrow $\beta$
  whose source (resp., target) is $y$ (resp., $x$) such that $\alpha\beta\in\I$ (resp., $\beta\alpha\in\I$).

  \item[(G3)] For each arrow $\alpha:x\to y$ in $\Q$, there is at most one arrow $\beta$
  whose source (resp., target) is $y$ (resp., $x$) such that $\alpha\beta\notin\I$ (resp., $\beta\alpha\notin\I$).

  \item[(G4)] $\I$ are generated by paths of length $2$.
\end{itemize}
\end{definition}

\begin{remark} \label{rem:hereditary-gentle-algebras}\rm
A {\defines hereditary gentle algebra} is a finite dimensional path algebra which is isomorphic to $\kk\Q$ such that the quiver $\Q$ is either of type $\A_n$ or of type $\tA_n$.
\end{remark}

In the sequel, we always assume that $A$ is a hereditary gentle algebra. Then there is a bijection between the set of isoclasses of hereditary gentle algebras and the set of homotopy classes of marked ribbon surfaces. Then we have a marked ribbon surface of $A$ and denote it by $\SURF(A)=(\Surf(A), \M(A), \Y(A), \Dgreen(A), \Dred(A))$. The following result is well-known.

\begin{proposition} \label{prop:disk-and-annulus}
The marked surface $\Surf(A)$ is either a disk or an annulus.
\end{proposition}

\begin{proof}
By Remark \ref{rem:hereditary-gentle-algebras}, if the quiver $\Q$ is of type $\A_n$, then $\Surf(A)$ is a disk and the number of $\gbullet$-marked points is $n+1$, where $n$ is the number of vertices of $\Q$ \cite[Corollary 1.23]{OPS2018}.
If the quiver $\Q$ of type $\tA_n$, then $\Surf(A)$ is an annulus \cite[Corollary 1.24]{OPS2018}.
\end{proof}

By Proposition \ref{prop:disk-and-annulus}, the (graded) marked ribbon surfaces of $A$ are the following two cases. Notice that we do not draw the foliations of $\SURF^{\F}(\kk\A_n)$ and $\SURF^{\F}(\kk\tA_n)$. However, we can find foliations and gradings of $\rbullet$-arcs such that $|\alpha|=0$ for any arrow $\alpha$ of $\kk\A_n$ or $\kk\tA_n$.

\newpage

\begin{center}
\begin{multicols}{2}%分两栏
\begin{figure}[H]
\begin{center}
\definecolor{ffqqqq}{rgb}{1,0,0}
\definecolor{qqwuqq}{rgb}{0,0.5,0}
\begin{tikzpicture}
% open FFAS
\draw [line width=1pt][rotate around={0:(0,0)}] (0,0) ellipse (3cm and 2cm);
\fill [qqwuqq] ( 0.00, 2.00) circle (2.5pt);
\fill [qqwuqq] (-3.00, 0.00) circle (2.5pt);
\fill [qqwuqq] (-2.82,-0.68) circle (2.5pt);
\fill [qqwuqq] (-1.5 ,-1.73) circle (2.5pt);
\draw [qqwuqq][line width=0.65pt] ( 0.00, 2.00)--(-3.00, 0.00);
\draw [qqwuqq][line width=0.65pt] ( 0.00, 2.00)--(-2.82,-0.68);
\draw [qqwuqq] (-1.7 ,-0.80) node{$\cdots$};
\draw [qqwuqq][line width=0.65pt] ( 0.00, 2.00)--(-1.5 ,-1.73);
\fill [qqwuqq] ( 1.02, 1.88) circle (2.5pt);
%\fill [qqwuqq] ( 1.93, 1.53) circle (2.5pt);
\fill [qqwuqq] ( 2.59, 1.  ) circle (2.5pt);
\draw [qqwuqq][line width=0.65pt] (-1.5 ,-1.73)--( 1.02, 1.88);
\draw [qqwuqq][line width=0.65pt] (-1.5 ,-1.73)--( 2.59, 1.  );
\draw [qqwuqq] ( 1.  , 0.80) node{$\cdots$};
\fill [qqwuqq] ( 3.  , 0.  ) circle (2.5pt);
\fill [qqwuqq] ( 0.  ,-2.  ) circle (2.5pt);
\fill [qqwuqq] ( 2.12,-1.41) circle (2.5pt);
\draw [qqwuqq][line width=0.65pt] ( 3.  , 0.  )--( 0.  ,-2.  );
\draw [qqwuqq][line width=0.65pt] ( 3.  , 0.  )--( 2.12,-1.41);
\draw [qqwuqq] ( 1.  ,-0.80) node{$\cdots$};
\draw [qqwuqq] ( 2.1 ,-1.  ) node{$\cdots$};
% closed FFAS
\draw [red][line width=0.65pt] ( 2.95, 0.35) to[out= 135, in=-135] ( 2.72, 0.85) [dash pattern=on 2pt off 2pt];
\draw [red][line width=0.65pt] ( 2.72, 0.85) to[out= 180, in=-90 ] ( 2.3 , 1.28);
\draw [red][line width=0.65pt] ( 2.3 , 1.28) to[out= 180, in=-40 ] ( 1.50, 1.73) [dash pattern=on 2pt off 2pt];
\draw [red][line width=0.65pt] ( 1.50, 1.73) to[out= 190, in=-30 ] ( 0.52, 1.97);
\fill [red] ( 2.95, 0.35) circle (2.5pt); \fill [white] ( 2.95, 0.35) circle (1.8pt);
\fill [red] ( 2.72, 0.85) circle (2.5pt); \fill [white] ( 2.72, 0.85) circle (1.8pt);
\fill [red] ( 2.3 , 1.28) circle (2.5pt); \fill [white] ( 2.3 , 1.28) circle (1.8pt);
\fill [red] ( 1.50, 1.73) circle (2.5pt); \fill [white] ( 1.50, 1.73) circle (1.8pt);
\fill [red] ( 0.52, 1.97) circle (2.5pt); \fill [white] ( 0.52, 1.97) circle (1.8pt);
\draw [red][line width=0.65pt] (-2.12, 1.41) -- (-2.95,-0.35);
\draw [red][line width=0.65pt] (-2.95,-0.35) to[out=-40 , in= 90 ] (-2.60,-1.  );
\draw [red][line width=0.65pt] (-2.60,-1.  ) to[out= 0  , in= 90 ] (-1.93,-1.53) [dash pattern=on 2pt off 2pt];
\draw [red][line width=0.65pt] (-1.93,-1.53) -- ( 0.52, 1.97);
\fill [red] (-2.12, 1.41) circle (2.5pt); \fill [white] (-2.12, 1.41) circle (1.8pt);
\fill [red] (-2.95,-0.35) circle (2.5pt); \fill [white] (-2.95,-0.35) circle (1.8pt);
\fill [red] (-2.60,-1.  ) circle (2.5pt); \fill [white] (-2.60,-1.  ) circle (1.8pt);
\fill [red] (-1.93,-1.53) circle (2.5pt); \fill [white] (-1.93,-1.53) circle (1.8pt);
\draw [red][line width=0.65pt] (-1.02,-1.87) -- ( 2.3 , 1.28);
\draw [red][line width=0.65pt] (-1.02,-1.87) to[out= 30 , in= 140] (-0.52,-1.97)[dash pattern=on 2pt off 2pt];
\draw [red][line width=0.65pt] (-0.52,-1.97) to[out= 30 , in= 150] ( 0.77,-1.93);
\draw [red][line width=0.65pt] ( 0.77,-1.93) to[out= 40 , in= 150] ( 1.5 ,-1.73)[dash pattern=on 2pt off 2pt];
\draw [red][line width=0.65pt] ( 1.5 ,-1.73) -- ( 2.81,-0.68);
\fill [red] (-1.02,-1.87) circle (2.5pt); \fill [white] (-1.02,-1.87) circle (1.8pt);
\fill [red] (-0.52,-1.97) circle (2.5pt); \fill [white] (-0.52,-1.97) circle (1.8pt);
\fill [red] ( 0.77,-1.93) circle (2.5pt); \fill [white] ( 0.77,-1.93) circle (1.8pt);
\fill [red] ( 1.5 ,-1.73) circle (2.5pt); \fill [white] ( 1.5 ,-1.73) circle (1.8pt);
\fill [red] ( 2.81,-0.68) circle (2.5pt); \fill [white] ( 2.81,-0.68) circle (1.8pt);
\end{tikzpicture}
\end{center}
\caption{$\SURF^{\F}(\kk\A_n)$. }
\label{fig:hga An}
\end{figure}

\begin{figure}[H]
\definecolor{ffqqqq}{rgb}{1,0,0}
\definecolor{qqwuqq}{rgb}{0,0.5,0}
\begin{tikzpicture} % [scale=1.3]
% \shade[left color=black!5] (0,0) circle (2cm);
\draw[line width=1.2pt] (0,0) circle (2cm);
\filldraw[color=black!20] (0,0) circle (1cm);
\draw[line width=1.2pt](0,0) circle (1cm);
% open FFAS
\fill [qqwuqq] ( 2.  , 0.  ) circle (2.5pt) [opacity=0.5];
\fill [qqwuqq] ( 1.  , 1.73) circle (2.5pt);
\fill [qqwuqq] ( 1.73, 1.  ) circle (2.5pt);
\fill [qqwuqq] ( 0.  , 2.  ) circle (2.5pt);
\fill [qqwuqq] (-1.  , 1.73) circle (2.5pt);
\fill [qqwuqq] (-1.73, 1.  ) circle (2.5pt);
\fill [qqwuqq] (-2.  , 0.  ) circle (2.5pt);
\fill [qqwuqq] (-1.73,-1.  ) circle (2.5pt);
\fill [qqwuqq] (-1.  ,-1.73) circle (2.5pt);
\fill [qqwuqq] ( 0.  ,-2.  ) circle (2.5pt);
\fill [qqwuqq] ( 1.  ,-1.73) circle (2.5pt);
\fill [qqwuqq] ( 1.73,-1.  ) circle (2.5pt);
\fill [qqwuqq] (0,1) circle (2.5pt)  [opacity=0];
\fill [qqwuqq] (-1.41*0.5,1.41*0.5) circle (2.5pt);
\fill [qqwuqq] (-1,0) circle (2.5pt) [opacity=0];
\fill [qqwuqq] (-1.41*0.5,-1.41*0.5) circle (2.5pt);
\fill [qqwuqq] (0,-1) circle (2.5pt) [opacity=0];
\fill [qqwuqq] (1.41*0.5,-1.41*0.5) circle (2.5pt);
\fill [qqwuqq] (1,0) circle (2.5pt)  [opacity=0];
\fill [qqwuqq] (1.41*0.5,1.41*0.5) circle (2.5pt);
\draw[qqwuqq][line width=0.65pt][dotted][opacity=0.5] (0,2)--(0,1);
\draw[qqwuqq][line width=0.65pt] (0,2)--(-1.41*0.5,1.41*0.5);
\draw[qqwuqq][line width=0.65pt] (0,2)--(1.41*0.5,1.41*0.5);
\draw[qqwuqq][line width=0.65pt][dotted][opacity=0.5] (-2,0)--(-1,0);
\draw[qqwuqq][line width=0.65pt] (-2,0)--(-1.41*0.5,1.41*0.5);
\draw[qqwuqq][line width=0.65pt] (-2,0)--(-1.41*0.5,-1.41*0.5);
\draw[qqwuqq][line width=0.65pt][dotted][opacity=0.5] (0,-2)--(0,-1);
\draw[qqwuqq][line width=0.65pt] (0,-2)--(-1.41*0.5,-1.41*0.5);
\draw[qqwuqq][line width=0.65pt] (0,-2)--(1.41*0.5,-1.41*0.5);
\draw[qqwuqq][line width=0.65pt][dotted][opacity=0.5] (2,0)--(1,0);
\draw[qqwuqq][line width=0.65pt][dotted][opacity=0.5] (2,0)--(1.41*0.5,1.41*0.5);
\draw[qqwuqq][line width=0.65pt][dotted][opacity=0.5] (2,0)--(1.41*0.5,-1.41*0.5);
\draw[qqwuqq][line width=0.65pt][dotted][opacity=0.5] (1.41*0.5,1.41*0.5) -- (1.41,1.41);
\draw[qqwuqq][line width=0.65pt] (1.41*0.5,1.41*0.5) -- (1.73,1);
\draw[qqwuqq][line width=0.65pt] (1.41*0.5,1.41*0.5) -- (1,1.73);
\draw[qqwuqq][line width=0.65pt][dotted][opacity=0.5] (-1.41*0.5,1.41*0.5) -- (-1.41,1.41);
\draw[qqwuqq][line width=0.65pt] (-1.41*0.5,1.41*0.5) -- (-1.73,1);
\draw[qqwuqq][line width=0.65pt] (-1.41*0.5,1.41*0.5) -- (-1,1.73);
\draw[qqwuqq][line width=0.65pt][dotted][opacity=0.5] (-1.41*0.5,-1.41*0.5) -- (-1.41,-1.41);
\draw[qqwuqq][line width=0.65pt] (-1.41*0.5,-1.41*0.5) -- (-1.73,-1);
\draw[qqwuqq][line width=0.65pt] (-1.41*0.5,-1.41*0.5) -- (-1,-1.73);
\draw[qqwuqq][line width=0.65pt][dotted][opacity=0.5] (1.41*0.5,-1.41*0.5) -- (1.41,-1.41);
\draw[qqwuqq][line width=0.65pt] (1.41*0.5,-1.41*0.5) -- (1.73,-1);
\draw[qqwuqq][line width=0.65pt] (1.41*0.5,-1.41*0.5) -- (1,-1.73);
% closed FFAS
\draw [red][line width=0.65pt] ( 1.93, 0.51) to[out= 150,in= -90] ( 1.58, 1.22);
\draw [red][line width=0.65pt] ( 1.58, 1.22) to[out= 180,in= -90] ( 1.22, 1.58) [dash pattern=on 2pt off 2pt];
\draw [red][line width=0.65pt] ( 1.22, 1.58) to[out= 180,in= -60] ( 0.52, 1.93); % I
\draw [red][line width=0.65pt] (-1.93, 0.51) to[out=  30,in= -90] (-1.58, 1.22);
\draw [red][line width=0.65pt] (-1.58, 1.22) to[out=   0,in= -90] (-1.22, 1.58) [dash pattern=on 2pt off 2pt];
\draw [red][line width=0.65pt] (-1.22, 1.58) to[out=   0,in=-120] (-0.52, 1.93); % II
\draw [red][line width=0.65pt] (-1.93,-0.51) to[out= -30,in=  90] (-1.58,-1.22);
\draw [red][line width=0.65pt] (-1.58,-1.22) to[out=   0,in=  90] (-1.22,-1.58) [dash pattern=on 2pt off 2pt];
\draw [red][line width=0.65pt] (-1.22,-1.58) to[out=   0,in= 120] (-0.52,-1.93); % III
\draw [red][line width=0.65pt] ( 1.93,-0.51) to[out=-150,in=  90] ( 1.58,-1.22);
\draw [red][line width=0.65pt] ( 1.58,-1.22) to[out=-180,in=  90] ( 1.22,-1.58) [dash pattern=on 2pt off 2pt];
\draw [red][line width=0.65pt] ( 1.22,-1.58) to[out=-180,in= 120] ( 0.52,-1.93); % IV
\draw [red][line width=0.65pt] ( 1.93, 0.51) -- ( 0.92, 0.38) [dash pattern=on 2pt off 2pt];
\draw [red][line width=0.65pt] ( 0.51, 1.93) -- ( 0.38, 0.92);
\draw [red][line width=0.65pt] (-0.51, 1.93) -- (-0.38, 0.92);
\draw [red][line width=0.65pt] (-1.93, 0.51) -- (-0.92, 0.38);
\draw [red][line width=0.65pt] (-1.93,-0.51) -- (-0.92,-0.38);
\draw [red][line width=0.65pt] (-0.51,-1.93) -- (-0.38,-0.92);
\draw [red][line width=0.65pt] ( 0.51,-1.93) -- ( 0.38,-0.92);
\draw [red][line width=0.65pt] ( 1.93,-0.51) -- ( 0.92,-0.38) [dash pattern=on 2pt off 2pt];
\draw [red][line width=0.65pt] ( 0.38, 0.92) to[out=  90,in=  90] (-0.38, 0.92) [dash pattern=on 2pt off 2pt];
\draw [red][line width=0.65pt] (-0.92, 0.38) to[out= 180,in= 180] (-0.92,-0.38) [dash pattern=on 2pt off 2pt];
\draw [red][line width=0.65pt] (-0.38,-0.92) to[out= -90,in= -90] ( 0.38,-0.92) [dash pattern=on 2pt off 2pt];
\draw [red][line width=0.65pt] ( 0.92,-0.38) to[out=   0,in=   0] ( 0.92, 0.38) [dash pattern=on 2pt off 2pt];
% outer
\fill [red] ( 1.93, 0.51) circle (2.5pt); \fill [white] ( 1.93, 0.51) circle (1.8pt); % 1- 15
\fill [red] ( 1.58, 1.22) circle (2.5pt); \fill [white] ( 1.58, 1.22) circle (1.8pt); % 2- 37.5
\fill [red] ( 1.22, 1.58) circle (2.5pt); \fill [white] ( 1.22, 1.58) circle (1.8pt); % 3- 52.5
\fill [red] ( 0.52, 1.93) circle (2.5pt); \fill [white] ( 0.52, 1.93) circle (1.8pt); % 4- 67.5
\fill [red] (-0.52, 1.93) circle (2.5pt); \fill [white] (-0.52, 1.93) circle (1.8pt); % 4- 67.5
\fill [red] (-1.22, 1.58) circle (2.5pt); \fill [white] (-1.22, 1.58) circle (1.8pt); % 3- 127.5
\fill [red] (-1.58, 1.22) circle (2.5pt); \fill [white] (-1.58, 1.22) circle (1.8pt); % 2- 140.5
\fill [red] (-1.93, 0.51) circle (2.5pt); \fill [white] (-1.93, 0.51) circle (1.8pt); % 1- 165
\fill [red] (-1.93,-0.51) circle (2.5pt); \fill [white] (-1.93,-0.51) circle (1.8pt); % 1-
\fill [red] (-1.58,-1.22) circle (2.5pt); \fill [white] (-1.58,-1.22) circle (1.8pt); % 2-
\fill [red] (-1.22,-1.58) circle (2.5pt); \fill [white] (-1.22,-1.58) circle (1.8pt); % 3-
\fill [red] (-0.52,-1.93) circle (2.5pt); \fill [white] (-0.52,-1.93) circle (1.8pt); % 4-
\fill [red] ( 0.52,-1.93) circle (2.5pt); \fill [white] ( 0.52,-1.93) circle (1.8pt); % 4-
\fill [red] ( 1.22,-1.58) circle (2.5pt); \fill [white] ( 1.22,-1.58) circle (1.8pt); % 3-
\fill [red] ( 1.58,-1.22) circle (2.5pt); \fill [white] ( 1.58,-1.22) circle (1.8pt); % 2-
\fill [red] ( 1.93,-0.51) circle (2.5pt); \fill [white] ( 1.93,-0.51) circle (1.8pt); % 1-
% inner
\fill [red] ( 0.92, 0.38) circle (2.5pt); \fill [white] ( 0.92, 0.38) circle (1.8pt); % 1-
\fill [red] ( 0.38, 0.92) circle (2.5pt); \fill [white] ( 0.38, 0.92) circle (1.8pt); % 2-
\fill [red] (-0.38, 0.92) circle (2.5pt); \fill [white] (-0.38, 0.92) circle (1.8pt); % 2-
\fill [red] (-0.92, 0.38) circle (2.5pt); \fill [white] (-0.92, 0.38) circle (1.8pt); % 1-
\fill [red] (-0.92,-0.38) circle (2.5pt); \fill [white] (-0.92,-0.38) circle (1.8pt); % 1-
\fill [red] (-0.38,-0.92) circle (2.5pt); \fill [white] (-0.38,-0.92) circle (1.8pt); % 2-
\fill [red] ( 0.38,-0.92) circle (2.5pt); \fill [white] ( 0.38,-0.92) circle (1.8pt); % 2-
\fill [red] ( 0.92,-0.38) circle (2.5pt); \fill [white] ( 0.92,-0.38) circle (1.8pt); % 1-
\end{tikzpicture}
\caption{$\SURF^{\F}(\kk\tA_n)$}
\label{fig:hga tilde An}
\end{figure}
\end{multicols} % 分栏结束
\end{center}

%The indecomposable objects and morphisms between two indecomposable objects in $\modcat A$ and $\per A$ are clearly calculated by Butler-Ringel \cite{BR1987}, Crawley-Boevey \cite{CB1989}, Krause \cite{Kra1991} and Arnesen-Laking-Pauksztello \cite{ALP2016}.

Denote by $\PC(\SURF^{\F})$ the set of all permissible curves with endpoints lying in $\M\cup\E$ and by $\CC(\SURF^{\F})$ the set of all permissible curves without endpoints (up to homotopy). Denote by $\AC_{\m}(\SURF^{\F})$ the set of all admissible curves with endpoints lying in $\M$ and by $\AC_{\oslash}(\SURF^{\F})$ the set of all admissible curves without endpoints (up to homotopy). Then

\begin{theorem} \label{thm:OPS and BCS corresponding}
Let $\mathscr{J}$ be the set of all Jordan blocks with non-zero eigenvalue.
\begin{itemize}
  \item[\rm(1)] {\rm \cite[Theorems 3.8 and 3.9]{BCS2019}}
    There exists a bijection
    \[ \MM: \PC(\SURF^{\F_A}(A)) \cup (\CC(\SURF^{\F_A}(A))\times\mathscr{J}) \to \ind(\modcat A) \]
    between the set $\PC(\SURF^{\F_A}(A)) \cup (\CC(\SURF^{\F_A}(A))\times\mathscr{J})$ of all permissible curves and the set $\ind(\modcat A)$ of all isoclasses of indecomposable modules in $\modcat A$.

  \item[\rm(2)] {\rm \cite[Theorem 2.12]{OPS2018}}
    There exists a bijection
    \[\X: \AC_{\m}(\SURF^{\F_A}(A)) \cup (\AC_{\oslash}(\SURF^{\F_A}(A)) \times \mathscr{J}) \to \ind(\per A),\]
    between the set $\AC_{\m}(\SURF^{\F_A}(A)) \cup (\AC_{\oslash}(\SURF^{\F_A}(A)) \times \mathscr{J})$ and the set of all isoclasses of indecomposable objects in $\per A$.
\end{itemize}
\end{theorem}

The indecomposable objects in $\modcat A$ (resp., $\per A$) corresponding to the curves in $\PC(\SURF^{\F_A}(A))$ (resp., $\AC_{\m}(\SURF^{\F_A}(A))$) are called {\defines string};

The indecomposable objects in $\modcat A$ (resp., $\per A$) corresponding to the curves in $\CC(\SURF^{\F_A}(A)$ (resp., $\AC_{\oslash}(\SURF^{\F_A}(A))$) are called {\defines band}.

\section{Curve embedding from module categories to derived categories} \label{sect:embedding}

In this section, we study the curve embedding from module categories to derived categories and show that any permissible curve $c$ can be viewed as an admissible curve $\tc$ such that $\X(\tc)$ is a minimal projective representation of $\MM(c)$.

\subsection{Projective representation of string modules}
%\marginpar{\tiny\color{blue} Some notations for permissible curve and its segments. }
Let $c$ be a permissible curve. Assume that $c$ consecutively crosses $\gbullet$-arcs $\agreen{c}{1}$, $\agreen{c}{2}$, $\ldots$, $\agreen{c}{m(c)}$ ($m(c)\in \NN^+$) in the interior of the surface. Denote by $p^c_i=c(s^c_i)$ ($0=s^c_0 < s^c_1 < \cdots < s^c_{m(c)} < s^c_{m(c)+1}=1$) the $i$-th intersection obtained by $c$ intersects with $\gbullet$-arcs and by $c_{(i,j)}$, say {\defines open arc segment} (=$\gbullet$-arc segment), the segment obtained by $\agreen{c}{i}$ and $\agreen{c}{j}$ cutting $c$. Then we have the following nine cases.

\begin{figure}[H]
\definecolor{ffqqqq}{rgb}{1,0,0}
\definecolor{qqwuqq}{rgb}{0,0.5,0}
\begin{tikzpicture}
% boundary
\draw[black] (-0.71-1,-0.71) node[left]{$\bSurf$};
\draw[black][line width=1.2pt] (-1,-1) arc (-90:-180:1);
\draw[black][line width=1.2pt] (-1,-1) -- ( 1,-1) [dotted];
\draw[black][line width=1.2pt] ( 1,-1) arc (-90:   0:1);
\draw[black][line width=1.2pt] ( 2, 0) arc (  0:  90:1) [dotted];
\draw[black][line width=1.2pt] ( 1, 1) -- (-1, 1) [dotted];
\draw[black][line width=1.2pt] (-1, 1) arc (  90: 180:1) [dotted];
% FFASgreen
\draw[qqwuqq] (-1,-1)--(-1, 1) [line width=1pt];
\fill[qqwuqq] (-1,-1) circle (0.1cm);
\fill[qqwuqq] (-1, 1) circle (0.1cm);
\draw[qqwuqq] ( 1,-1)--( 1, 1) [line width=1pt];
\fill[qqwuqq] ( 1,-1) circle (0.1cm);
\fill[qqwuqq] ( 1, 1) circle (0.1cm);
\draw[qqwuqq] (-1, 1)--(-2, 0) [line width=1pt];
\fill[qqwuqq] (-1, 1) circle (0.1cm);
\fill[qqwuqq] (-2, 0) circle (0.1cm);
\draw[qqwuqq] ( 1, 1)--( 2, 0) [line width=1pt];
\fill[qqwuqq] ( 2, 0) circle (0.1cm);
% permissible curve
\draw[blue][line width=1pt](-2,0) to[out=45,in=-135] (2,0);
%[postaction={on each segment={mid arrow=blue}}]
\draw[blue] (0,0) node[above]{$c$};
% subcaption
\draw (0,-1.5) node{Type 1.1};
\end{tikzpicture}
\ \
\begin{tikzpicture}
% boundary
\draw[black] (-0.71-1, 0.71) node[left]{$\bSurf$};
\draw[black][line width=1.2pt] (-1,-1) arc (-90:-180:1)[dotted];
\draw[black][line width=1.2pt] (-1,-1) -- ( 1,-1)[dotted];
\draw[black][line width=1.2pt] ( 1,-1) arc (-90:   0:1);
\draw[black][line width=1.2pt] ( 2, 0) arc (  0:  90:1)[dotted];
\draw[black][line width=1.2pt] ( 1, 1) -- (-1, 1) [dotted];
\draw[black][line width=1.2pt] (-1, 1) arc (  90: 180:1);
% FFASgreen
\draw[qqwuqq] (-1,-1)--(-1, 1) [line width=1pt];
\fill[qqwuqq] (-1,-1) circle (0.1cm);
\fill[qqwuqq] (-1, 1) circle (0.1cm);
\draw[qqwuqq] ( 1,-1)--( 1, 1) [line width=1pt];
\fill[qqwuqq] ( 1,-1) circle (0.1cm);
\fill[qqwuqq] ( 1, 1) circle (0.1cm);
\draw[qqwuqq] (-1,-1)--(-2, 0) [line width=1pt];
\fill[qqwuqq] (-1, 1) circle (0.1cm);
\fill[qqwuqq] (-2, 0) circle (0.1cm);
\draw[qqwuqq] ( 1, 1)--( 2, 0) [line width=1pt];
\fill[qqwuqq] ( 2, 0) circle (0.1cm);
% permissible curve
\draw[blue][line width=1pt](-2,0) to[out=45,in=-135] (2,0);
%[postaction={on each segment={mid arrow=blue}}]
\draw[blue] (0,0) node[above]{$c$};
% subcaption
\draw (0,-1.5) node{Type 1.2};
\end{tikzpicture}
\ \
\begin{tikzpicture}
% boundary
\draw[black] (-0.71-1,-0.71) node[left]{$\bSurf$};
\draw[black][line width=1.2pt] (-1,-1) arc (-90:-180:1);
\draw[black][line width=1.2pt] (-1,-1) -- ( 1,-1)[dotted];
\draw[black][line width=1.2pt] ( 1,-1) arc (-90:   0:1)[dotted];
\draw[black][line width=1.2pt] ( 2, 0) arc (  0:  90:1);
\draw[black][line width=1.2pt] ( 1, 1) -- (-1, 1) [dotted];
\draw[black][line width=1.2pt] (-1, 1) arc (  90: 180:1) [dotted];
% FFASgreen
\draw[qqwuqq] (-1,-1)--(-1, 1) [line width=1pt];
\fill[qqwuqq] (-1,-1) circle (0.1cm);
\fill[qqwuqq] (-1, 1) circle (0.1cm);
\draw[qqwuqq] ( 1,-1)--( 1, 1) [line width=1pt];
\fill[qqwuqq] ( 1,-1) circle (0.1cm);
\fill[qqwuqq] ( 1, 1) circle (0.1cm);
\draw[qqwuqq] (-1, 1)--(-2, 0) [line width=1pt];
\fill[qqwuqq] (-1, 1) circle (0.1cm);
\fill[qqwuqq] (-2, 0) circle (0.1cm);
\draw[qqwuqq] ( 1,-1)--( 2, 0) [line width=1pt];
\fill[qqwuqq] ( 2, 0) circle (0.1cm);
% permissible curve
\draw[blue][line width=1pt](-2,0) to[out=45,in=-135] (2,0);
%[postaction={on each segment={mid arrow=blue}}]
\draw[blue] (0,0) node[above]{$c$};
% subcaption
\draw (0,-1.5) node{Type 1.3};
\end{tikzpicture}
\caption{Type 1: $c(0),c(1)\in \M$. }
\label{fig:PC type 1}
\end{figure}

\begin{figure}[H]
\definecolor{ffqqqq}{rgb}{1,0,0}
\definecolor{qqwuqq}{rgb}{0,0.5,0}
\begin{tikzpicture}
% boundary
\draw[black] ( 0.71+1,-0.71) node[right]{$\bSurf$};
\draw[black][line width=1.2pt] (-1,-1) arc (-90:-180:1) [dotted];
\draw[black][line width=1.2pt] (-1,-1) -- ( 1,-1) [dotted];
\draw[black][line width=1.2pt] ( 1,-1) arc (-90:   0:1);
\draw[black][line width=1.2pt] ( 2, 0) arc (  0:  90:1) [dotted];
\draw[black][line width=1.2pt] ( 1, 1) -- (-1, 1) [dotted];
\draw[black][line width=1.2pt] (-1, 1) arc (  90: 180:1);
% FFASgreen
\draw[qqwuqq] (-1,-1)--(-1, 1) [line width=1pt];
\fill[qqwuqq] (-1,-1) circle (0.1cm);
\fill[qqwuqq] (-1, 1) circle (0.1cm);
\draw[qqwuqq] ( 1,-1)--( 1, 1) [line width=1pt];
\fill[qqwuqq] ( 1,-1) circle (0.1cm);
\fill[qqwuqq] ( 1, 1) circle (0.1cm);
\draw[qqwuqq] (-1, 1)--(-2, 0) [line width=1pt];
\fill[qqwuqq] (-1, 1) circle (0.1cm);
\fill[qqwuqq] (-2, 0) circle (0.1cm);
\draw[qqwuqq] ( 1, 1)--( 2, 0) [line width=1pt];
\fill[qqwuqq] ( 2, 0) circle (0.1cm);
% FFASred
\draw[red][line width=0.7pt] (-0.7-1,+0.7) circle (0.09cm);
% permissible curve
\draw[blue][line width=1pt](-0.7-1,+0.7) to[out=-45,in=-135] (2,0);
%[postaction={on each segment={mid arrow=blue}}]
\draw[blue] (0,-0.45) node[below]{$c$};
% subcaption
\draw (0,-1.5) node{Type 2.1};
\end{tikzpicture}
\ \
\begin{tikzpicture}
% boundary
\draw[black] (0.71+1,-0.71) node[right]{$\bSurf$};
\draw[black][line width=1.2pt] (-1,-1) arc (-90:-180:1);
\draw[black][line width=1.2pt] (-1,-1) -- ( 1,-1) [dotted];
\draw[black][line width=1.2pt] ( 1,-1) arc (-90:   0:1);
\draw[black][line width=1.2pt] ( 2, 0) arc (  0:  90:1) [dotted];
\draw[black][line width=1.2pt] ( 1, 1) -- (-1, 1) [dotted];
\draw[black][line width=1.2pt] (-1, 1) arc (  90: 180:1);
% FFASgreen
\draw[qqwuqq] (-1,-1)--(-1, 1) [line width=1pt];
\fill[qqwuqq] (-1,-1) circle (0.1cm);
\fill[qqwuqq] (-1, 1) circle (0.1cm);
\draw[qqwuqq] ( 1,-1)--( 1, 1) [line width=1pt];
\fill[qqwuqq] ( 1,-1) circle (0.1cm);
\fill[qqwuqq] ( 1, 1) circle (0.1cm);
\draw[qqwuqq] (-1,-1)--(-2, 0) [line width=1pt];
\fill[qqwuqq] (-1, 1) circle (0.1cm);
\fill[qqwuqq] (-2, 0) circle (0.1cm);
\draw[qqwuqq] ( 1, 1)--( 2, 0) [line width=1pt];
\fill[qqwuqq] ( 2, 0) circle (0.1cm);
% FFASred
\draw[red][line width=0.7pt] (-0.7-1,-0.7) circle (0.09cm);
% permissible curve
\draw[blue][line width=1pt](-0.7-1,-0.7) to[out=45,in=-135] (2,0);
%[postaction={on each segment={mid arrow=blue}}]
\draw[blue] (0,-0.4) node[above]{$c$};
% subcaption
\draw (0,-1.5) node{Type 2.2};
\end{tikzpicture}
\ \
\begin{tikzpicture}
% boundary
\draw[black] ( 0.71+1,+0.71) node[right]{$\bSurf$};
\draw[black][line width=1.2pt] (-1,-1) arc (-90:-180:1);
\draw[black][line width=1.2pt] (-1,-1) -- ( 1,-1)[dotted];
\draw[black][line width=1.2pt] ( 1,-1) arc (-90:   0:1)[dotted];
\draw[black][line width=1.2pt] ( 2, 0) arc (  0:  90:1);
\draw[black][line width=1.2pt] ( 1, 1) -- (-1, 1) [dotted];
\draw[black][line width=1.2pt] (-1, 1) arc (  90: 180:1);
% FFASgreen
\draw[qqwuqq] (-1,-1)--(-1, 1) [line width=1pt];
\fill[qqwuqq] (-1,-1) circle (0.1cm);
\fill[qqwuqq] (-1, 1) circle (0.1cm);
\draw[qqwuqq] ( 1,-1)--( 1, 1) [line width=1pt];
\fill[qqwuqq] ( 1,-1) circle (0.1cm);
\fill[qqwuqq] ( 1, 1) circle (0.1cm);
\draw[qqwuqq] (-1, 1)--(-2, 0) [line width=1pt];
\fill[qqwuqq] (-1, 1) circle (0.1cm);
\fill[qqwuqq] (-2, 0) circle (0.1cm);
\draw[qqwuqq] ( 1,-1)--( 2, 0) [line width=1pt];
\fill[qqwuqq] ( 2, 0) circle (0.1cm);
% FFASred
\draw[red][line width=0.7pt] (-0.7-1,+0.7) circle (0.09cm);
% permissible curve
\draw[blue][line width=1pt](-0.7-1,+0.7) to[out=-45,in=-135] (2,0);
%[postaction={on each segment={mid arrow=blue}}]
\draw[blue] (0,-0.4) node[above]{$c$};
% subcaption
\draw (0,-1.5) node{Type 2.3};
\end{tikzpicture}
\caption{Type 2: $c(0)\in \E$ and $c(1)\in \M$.}
%(Note that the bijection $\MM$ is independent on the direction of permissible curves,
%thus the case of $c(0)\in \M$ and $c(1)\in \E$ is dual.) }
\label{fig:PC type 2}
\end{figure}

\begin{figure}[H]
\definecolor{ffqqqq}{rgb}{1,0,0}
\definecolor{qqwuqq}{rgb}{0,0.5,0}
\begin{tikzpicture}
% boundary
\draw[black] ( 0.71+1,-0.71) node[right]{$\bSurf$};
\draw[black][line width=1.2pt] (-1,-1) arc (-90:-180:1);
\draw[black][line width=1.2pt] (-1,-1) -- ( 1,-1) [dotted];
\draw[black][line width=1.2pt] ( 1,-1) arc (-90:   0:1);
\draw[black][line width=1.2pt] ( 2, 0) arc (  0:  90:1);
\draw[black][line width=1.2pt] ( 1, 1) -- (-1, 1) [dotted];
\draw[black][line width=1.2pt] (-1, 1) arc (  90: 180:1);
% FFASgreen
\draw[qqwuqq] (-1,-1)--(-1, 1) [line width=1pt];
\fill[qqwuqq] (-1,-1) circle (0.1cm);
\fill[qqwuqq] (-1, 1) circle (0.1cm);
\draw[qqwuqq] ( 1,-1)--( 1, 1) [line width=1pt];
\fill[qqwuqq] ( 1,-1) circle (0.1cm);
\fill[qqwuqq] ( 1, 1) circle (0.1cm);
\draw[qqwuqq] (-1, 1)--(-2, 0) [line width=1pt];
\fill[qqwuqq] (-1, 1) circle (0.1cm);
\fill[qqwuqq] (-2, 0) circle (0.1cm);
\draw[qqwuqq] ( 1, 1)--( 2, 0) [line width=1pt];
\fill[qqwuqq] ( 2, 0) circle (0.1cm);
% FFASred
\draw[red][line width=0.7pt] (-0.7-1,+0.7) circle (0.09cm);
\draw[red][line width=0.7pt] (+0.7+1,+0.7) circle (0.09cm);
% permissible curve
\draw[blue][line width=1pt](-0.7-1,+0.7) to[out=-45,in=-135] (0.7+1,+0.7);
%[postaction={on each segment={mid arrow=blue}}]
\draw[blue] (0,0) node[below]{$c$};
% subcaption
\draw (0,-1.5) node{Type 3.1};
\end{tikzpicture}
\ \
\begin{tikzpicture}
% boundary
\draw[black] (0.71+1,-0.71) node[right]{$\bSurf$};
\draw[black][line width=1.2pt] (-1,-1) arc (-90:-180:1);
\draw[black][line width=1.2pt] (-1,-1) -- ( 1,-1) [dotted];
\draw[black][line width=1.2pt] ( 1,-1) arc (-90:   0:1);
\draw[black][line width=1.2pt] ( 2, 0) arc (  0:  90:1);
\draw[black][line width=1.2pt] ( 1, 1) -- (-1, 1) [dotted];
\draw[black][line width=1.2pt] (-1, 1) arc (  90: 180:1);
% FFASgreen
\draw[qqwuqq] (-1,-1)--(-1, 1) [line width=1pt];
\fill[qqwuqq] (-1,-1) circle (0.1cm);
\fill[qqwuqq] (-1, 1) circle (0.1cm);
\draw[qqwuqq] ( 1,-1)--( 1, 1) [line width=1pt];
\fill[qqwuqq] ( 1,-1) circle (0.1cm);
\fill[qqwuqq] ( 1, 1) circle (0.1cm);
\draw[qqwuqq] (-1,-1)--(-2, 0) [line width=1pt];
\fill[qqwuqq] (-1, 1) circle (0.1cm);
\fill[qqwuqq] (-2, 0) circle (0.1cm);
\draw[qqwuqq] ( 1, 1)--( 2, 0) [line width=1pt];
\fill[qqwuqq] ( 2, 0) circle (0.1cm);
% FFASred
\draw[red][line width=0.7pt] (-0.7-1,-0.7) circle (0.09cm);
\draw[red][line width=0.7pt] (+0.7+1,+0.7) circle (0.09cm);
% permissible curve
\draw[blue][line width=1pt](-0.7-1,-0.7) to[out=45,in=-135] (0.7+1,+0.7);
%[postaction={on each segment={mid arrow=blue}}]
\draw[blue] (0,-0) node[above]{$c$};
% subcaption
\draw (0,-1.5) node{Type 3.2};
\end{tikzpicture}
\ \
\begin{tikzpicture}
% boundary
\draw[black] ( 0.71+1,+0.71) node[right]{$\bSurf$};
\draw[black][line width=1.2pt] (-1,-1) arc (-90:-180:1);
\draw[black][line width=1.2pt] (-1,-1) -- ( 1,-1)[dotted];
\draw[black][line width=1.2pt] ( 1,-1) arc (-90:   0:1);
\draw[black][line width=1.2pt] ( 2, 0) arc (  0:  90:1);
\draw[black][line width=1.2pt] ( 1, 1) -- (-1, 1) [dotted];
\draw[black][line width=1.2pt] (-1, 1) arc (  90: 180:1);
% FFASgreen
\draw[qqwuqq] (-1,-1)--(-1, 1) [line width=1pt];
\fill[qqwuqq] (-1,-1) circle (0.1cm);
\fill[qqwuqq] (-1, 1) circle (0.1cm);
\draw[qqwuqq] ( 1,-1)--( 1, 1) [line width=1pt];
\fill[qqwuqq] ( 1,-1) circle (0.1cm);
\fill[qqwuqq] ( 1, 1) circle (0.1cm);
\draw[qqwuqq] (-1, 1)--(-2, 0) [line width=1pt];
\fill[qqwuqq] (-1, 1) circle (0.1cm);
\fill[qqwuqq] (-2, 0) circle (0.1cm);
\draw[qqwuqq] ( 1,-1)--( 2, 0) [line width=1pt];
\fill[qqwuqq] ( 2, 0) circle (0.1cm);
% FFASred
\draw[red][line width=0.7pt] (-0.7-1,+0.7) circle (0.09cm);
\draw[red][line width=0.7pt] (+0.7+1,-0.7) circle (0.09cm);
% permissible curve
\draw[blue][line width=1pt](-0.7-1,+0.7) to[out=-45,in=135] (+0.7+1,-0.7);
%[postaction={on each segment={mid arrow=blue}}]
\draw[blue] (0,0) node[below]{$c$};
% subcaption
\draw (0,-1.5) node{Type 3.3};
\end{tikzpicture}
\caption{Type 3: $c(0), c(1)\in \E$.}
\label{fig:PC type 3}
\end{figure}

%\marginpar{\tiny\color{blue} Some notations for admissible curve and its segments. }
%\textcolor[rgb]{1.00,0.00,0.00}{Similarly, assume that $\tc$ consecutively crosses grade $\rbullet$-arcs $\tared{\tc}{1}$, $\tared{\tc}{2}$, $\ldots$, $\tared{\tc}{n(\tc)}$ ($n(\tc)\in\NN^+$) in this paper.
%We denote by $q^{\tc}_i=c(t^{\tc}_i)$ ($0=t^{\tc}_0 < t^{\tc}_1 < \ldots < t^{\tc}_{n(\tc)} < t^{\tc}_{n(\tc)+1}=1$) the $i$-th intersection obtained by $c$ intersects with $\Dred$;
%$d^{\tc}_i$ the intersection index $\ii_{q^{\tc}_i}(\tc, \tared{\tc}{i})$;
%and $c_{[i,j]}$, say {\defines closed arc segment} (=$\rbullet$-arc segment), the segment obtained by $\tared{\tc}{i}$ and $\tared{\tc}{j}$ cutting $\tc$.}

%In order to calculate the admissible curve corresponding to the resolution of string module, we need divide all permissible curves in $\PC(\SURF^{\F_A}(A))$ to three classes (six types) shown in \Pic \ref{fig:PC type 1}, \ref{fig:PC type 2} and \ref{fig:PC type 3}
%(we ignore the $\rbullet$-marked points sets $\Y$, the $\rbullet$-FFASs $\Dred$ and the foliations $\F_A$ of marked surfaces in figures).

In order to characterize the projective representation of  $\MM(c)$, we need the following definition.

\begin{definition}\rm
Let $x$ be the common endpoint of $\agreen{c}{i-1}$ and $\agreen{c}{i}$ and $y$ be the common endpoint of $\agreen{c}{i}$ and $\agreen{c}{i+1}$. Then $\agreen{c}{i}$ is called a {\defines top $($resp., socle$)$ $\gbullet$-arc of $c$} if $x$ is on the right (resp., left) of $c$ and $y$ is on the left (resp., right) of $c$. Denote by $\top(c)$ (resp., $\soc(c)$) the set of all top (resp., socle) $\gbullet$-arcs of $c$.
\end{definition}

In particular, we have

\begin{remark} \rm
\begin{itemize}
  \item[(1)] If $c$ crosses only one $\gbullet$-arc, then this $\gbullet$-arc is not only a top but also a socle $\gbullet$-arc of $c$.
  \item[(2)] If $i=m(c)$ or $1$ and $x$ is on the right (resp., left) of $c$ and $y$ is on the left (resp., right) of $c$, then $\agreen{c}{i}$ is a top (resp., socle) $\gbullet$-arc of $c$.
\end{itemize}
\end{remark}

Let $c_{\simp}(a_{\gbullet})$ be the permissible curve crossing only one $\gbullet$-arc $a_{\gbullet}$. Then $\MM(c_{\simp}(a_{\gbullet}))$ is simple. Thus, we have
\begin{align}
  \MM(\top (c)) & := \bigoplus_{a_{\gbullet}\in\top(c)}\MM(c_{\simp}(a_{\gbullet})) \cong \top\MM(c); \label{formula:top and soc 1} \\
  \MM(\soc (c)) & := \bigoplus_{a_{\gbullet}\in\soc(c)}\MM(c_{\simp}(a_{\gbullet})) \cong \soc\MM(c). \label{formula:top and soc 2}
\end{align}

\begin{definition}\rm
We say the permissible curve $c$, denoted by $c_{\proj}(a_{\gbullet})$, is a {\defines projective curve corresponding to $a_{\gbullet}=\agreen{c}{i}$}, if there is a unique integer $1\le i\le m(c)$ satisfying the following conditions:
\begin{itemize}
  \item[(i)]
    $\agreen{c}{1}$, $\agreen{c}{2}$, $\ldots$, $\agreen{c}{i}$ has a common endpoint $x$ which is right to $c$, and
        there is no arc $a\in\Dgreen$ with endpoint $x$ such that $a$ is left to $\agreen{c}{1}$ at the point $x$;
  \item[(ii)]
    $\agreen{c}{i}$, $\agreen{c}{i+1}$, $\ldots$, $\agreen{c}{m(c)}$ has a common endpoint $y$ which is left to $c$, and there is no arc $\hat{a}\in\Dgreen$ with endpoint $y$ such that $\hat{a}$ is left to $\agreen{c}{m(c)}$ at the point $y$.
\end{itemize}
\end{definition}

It is easy to see that $\MM(c_{\proj}(a_{\gbullet}))$ is an indecomposable projective module.
For any $\gbullet$-arc $a_{\gbullet}$, we denote by $\overleftarrow{a_{\gbullet}}$ the $\gbullet$-arc $a_{\gbullet}'$ which is left to $a_{\gbullet}$ at the intersection of $a_{\gbullet}\cap a_{\gbullet}'$. In particular, if there are no $\gbullet$-arcs left to $a_{\gbullet}$, then we
write $\overleftarrow{a_{\gbullet}}=\varnothing$ and say $c_{\proj}(\overleftarrow{a_{\gbullet}})$ is a trivial permissible curve. In this case, $\MM(c_{\proj}(\overleftarrow{a_{\gbullet}}))=0$.

\begin{lemma} \label{lemm:proj cover}
Let $c$ be a permissible curve of $\SURF(A)$ and $\widehat{\soc(c)}$ be the set of all socle $\gbullet$-arcs of $c$ except $c_{\simp}(\agreen{c}{1})$ and $c_{\simp}(\agreen{c}{m(c)})$
{\footnote{Notice that $c_{\simp}(\agreen{c}{1})$ or $c_{\simp}(\agreen{c}{m(c)})$ may not in $\soc(c)$. }}.
Then the projective cover of $\MM(c)$ is
\[ p_0: P_0 = \bigoplus_{a_{\gbullet}\in\top(c)}\MM(c_{\proj}(a_{\gbullet})) \to \MM(c) \]
and the kernel $\ker p_0$ of $p_0$ is isomorphic to $Q_L\oplus Q\oplus Q_R$, where:
\[Q_L \cong \MM(c_{\proj}(\overleftarrow{\agreen{c}{1}})),\
Q \cong \bigoplus\limits_{a_{\gbullet}\in \widehat{\soc(c)}} \MM(c_{\proj}(a_{\gbullet})),\
Q_R \cong \MM(c_{\proj}(\overleftarrow{\agreen{c}{m(c)}})).\]
\end{lemma}

\begin{proof}
Let $v_i$ ($1\le i\le m(c)$) be the vertex corresponding to the $\gbullet$-arc $\agreen{c}{i}$ in $\Q(\SURF^{\F})$.
Then, by Theorem \ref{thm:OPS and BCS corresponding} (1), $\MM(c)$ is a string module decided by the following string (i.e., a path of the underlying graph of $\Q$ without relations, whose arrows are decided by $c$ and $\gbullet$-FFAS):
\begin{align}\label{formula:lemm-proj cover}
v_1 \longline v_2 \longline \cdots \longline v_{m(c)-1} \longline v_{m(c)}
\end{align}
Let $v_i^{\top}$ (resp., $v_i^{\soc}$) be the vertex corresponding to the $\gbullet$-arc $\agreen{c}{i} \in \top(c)$ (resp., $\in \soc(c)$).
Then, by (\ref{formula:top and soc 1}), we obtain the projective cover of $\MM(c)$ is $p_0: P_0 \to \MM(c)$, where
\[ P_0 \cong \bigoplus_{\agreen{c}{i} \in \top(c)} P(v_i^{\top}) \cong \bigoplus_{\agreen{c}{i} \in \top(c)} \MM(c_{\proj}(\agreen{c}{i})). \]
Generally, there are four cases for the string (\ref{formula:lemm-proj cover}):
\begin{itemize}
  \item[(\textbf{a})] $\agreen{c}{1}$ is a socle $\gbullet$-arc of $c$, but $\agreen{c}{m(c)}$ is not;
  \item[(\textbf{b})] $\agreen{c}{m(c)}$ is a socle $\gbullet$-arc of $c$, but $\agreen{c}{1}$ is not;
  \item[(\textbf{c})] $\agreen{c}{1}$ and $\agreen{c}{m(c)}$ are socle $\gbullet$-arcs of $c$;
  \item[(\textbf{d})] $\agreen{c}{1}$ and $\agreen{c}{m(c)}$ are not socle $\gbullet$-arcs of $c$.
\end{itemize}
For the case (\textbf{a}), without loss of generality, we assume that the string (\ref{formula:lemm-proj cover}) is of the form
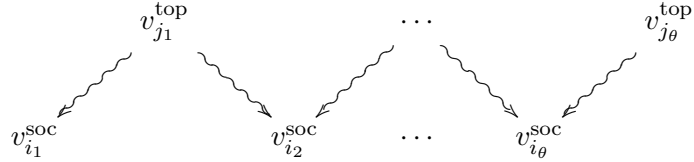
\[\xymatrix{
& v^{\top}_{j_1} \ar@{~>}[ld] \ar@{~>}[rd] & & \ar@{~>}[ld] \cdots \ar@{~>}[rd] & & v^{\top}_{j_{\theta}} \ar@{~>}[ld] \\
v^{\soc}_{i_1} & & v^{\soc}_{i_2} & \cdots & v^{\soc}_{i_{\theta}} &
}\]
where $v^{\soc}_{i_1}=v_1$, $v^{\top}_{j_\theta}=v_{m(c)}$ and ``$\xymatrix{v \ar@{~>}[r]& v'}$'' represents a path from $v$ to $v'$ in $\Q$. Since $A$ is hereditary gentle, we have
\[\widehat{\soc(c)} = \{\agreen{c}{i_t} \mid 2\le t\le \theta\} \]
and
\[\ker p_0 = P(x) \oplus \Big(\bigoplus\limits_{a_{\gbullet}\in \widehat{\soc(c)}} \MM(c_{\proj}(a_{\gbullet}))\Big) \oplus P(y),\]
where $x$ is a sink of the arrow whose source is $v_1$ and $y$ ia a sink of the arrow whose source is $v_{m(c)}$.
Then $x$ and $y$ are vertices corresponding to $\overleftarrow{\agreen{c}{1}}$ and $\overleftarrow{\agreen{c}{m(c)}}$, respectively. Thus
\[ P(x) \cong \MM(c_{\proj}(\overleftarrow{\agreen{c}{1}})) \ \text{and }\
   P(y) \cong \MM(c_{\proj}(\overleftarrow{\agreen{c}{m(c)}})). \]
The proof for cases (\textbf{b}), (\textbf{c}) and (\textbf{d}) are similar to (\textbf{a}).
\end{proof}

\begin{remark}\rm
Since $A$ is a gentle algebra, we have $Q$ is projective. Moreover, $A$ is hereditary, we obtain that $Q_L$ and $Q_R$ are projective.
\end{remark}

\begin{corollary} \label{coro:proj resolution}
The projective representation of arbitrary string module $\MM(c)$ in $\modcat A$ is of the form
\[ 0 \longrightarrow Q_L\oplus Q\oplus Q_R \longrightarrow
\bigoplus_{a_{\gbullet}\in\top(c)}\MM(c_{\proj}(a_{\gbullet}))  \longrightarrow \MM(c) \longrightarrow 0. \]
\end{corollary}

\subsection{The admissible curves corresponding to string complexes} \label{subsect:curves corresp. string comp.}

In this section, we study the embedding of curves from module categories to derived categories of hereditary gentle algebras.

First, we recall the definition of homotopy string for gentle algebras. Denote by $\alpha^-$ the {\defines formal inverse} of arrow $\alpha\in\Q_1$ and $\Q_1^-$ the set of all formal inverses of arrows.
Then, $s(\alpha^-)=t(\alpha)$ and $t(\alpha^-)=s(\alpha)$.
A {\defines direct $($resp., inverse$)$ homotopy letter} (of length $l$) is a sequence $\ell=u_1u_2\cdots u_l$,
where $t(u_i)=s(u_{i+1})$ for all $1\le i<l$ and all $u_i$ are elements in $\Q_1$ (resp., $\Q_1^-$).
A {\defines homotopy string} (of length $n$) is a sequence $s = \ell_1\cdots\ell_m$ ($l_1+\cdots+l_m=n$), where:
\begin{itemize}
  \item For any $1\le j\le m$, $\ell_j = u_{j,1}\cdots u_{j,l_j}$ is a homotopy letter.
  \item For any $1\le j < m$, $t(u_{j,l_j})=s(u_{j+1, 1})$.
  \item For arbitrary two homotopy letters $\ell_j$ and $\ell_{j+1}$, one of the following conditions holds:
    \begin{itemize}
     \item If $u_{j,l_j}\in \Q_1$ (resp., $\Q_1^-$), $u_{j+1,1}\in\Q_1^-$ (resp., $\Q_1$), then $u_{j,l_j}\ne u_{j+1,1}^-$;
     \item If $u_{j,l_j}$ and $u_{j+1,1}$ are elements in $\Q_1$ (resp., $\Q_1^-$) such that $t(u_{j,l_j}) = s(u_{j+1,1})$, then $u_{j,l_j}u_{j+1,1}\in \I$ (resp., $u_{j+1,1}^{-}u_{j,l_j}^{-}\in \I$).
    \end{itemize}
\end{itemize}
We say two homotopy strings $s$ and $s'$ are {\defines equivalence} if $s=s'$ or $s=s'^{-}$. There is a bijection between the homotopy strings and the indecomposable object in $\per A$ \cite[Theorem 2]{BM2003}.

For any permissible curve $c$ in $\PC(\SURF^{\F_A}(A))$, the end segments $c_{(0,1)}$ and $c_{(m(c), m(c)+1)}$ have the following three forms in Figure \ref{fig:endsegment}.

\begin{figure}[htbp]
\definecolor{ffqqqq}{rgb}{1,0,0}
\definecolor{qqwuqq}{rgb}{0,0.5,0}
\begin{tikzpicture}
% boundary
\draw[black] (-2,-1) node[below]{$\bSurf$};
\draw[black][<-][line width=1.2pt] (-2, 1) -- ( 2, 1);
\draw[black][->][line width=1.2pt] (-2,-1) -- ( 2,-1);
% FFASgreen
\draw[qqwuqq] ( 0,-1)--( 0, 1) [line width=1pt];
\fill[qqwuqq] ( 0,-1) circle (0.1cm);
\fill[qqwuqq] ( 0, 1) circle (0.1cm);
\draw[qqwuqq] ( 0,-1)--(-1, 1) [line width=1pt];
\fill[qqwuqq] (-1, 1) circle (0.1cm);
\draw (0,-0.25) node[right]{$\agreen{c}{1}$};
\draw (-0.25,-0.25) node[left]{$a_{\gbullet}=\overleftarrow{\agreen{c}{1}}$};
% FFASred
\draw [red] (-0.5,1) to[out= -45, in= 180] ( 0.5, 0.5) [line width=1pt];
\draw [red] (-0.5,1) to[out=-135, in=  45] (-1.3, 0.2) [line width=1pt];
\fill [red] (-0.5,1) circle (2.5pt); \fill [white] (-0.5,1) circle (1.8pt);
\draw (-1.3, 0.2) node[left]{$\ta_{\rbullet}$};
% permissible curve
\draw[blue][line width=1pt] (-1,1) to[out=-45,in=180] (2,0);
\draw[blue] (2,0) node[above]{$c$};
% subcaption
\draw (0,-1.5) node{Case $\bfI$};
\end{tikzpicture}
\ \
\begin{tikzpicture}
% boundary
\draw[black] (-2,-1) node[below]{$\bSurf$};
\draw[black][<-][line width=1.2pt] (-2, 1) -- ( 2, 1);
\draw[black][->][line width=1.2pt] (-2,-1) -- ( 2,-1);
% FFASgreen
\draw[qqwuqq] ( 0,-1)--( 0, 1) [line width=1pt];
\fill[qqwuqq] ( 0,-1) circle (0.1cm);
\fill[qqwuqq] ( 0, 1) circle (0.1cm);
\draw[qqwuqq] ( 0, 1)--(-1,-1) [line width=1pt];
\fill[qqwuqq] (-1,-1) circle (0.1cm);
\draw (0, 0.25) node[right]{$\agreen{c}{1}$};
\draw (-0.25,0.25) node[left]{$a_{\gbullet}$};
% FFASred
\draw [red] (-0.5,-1) to[out=  45, in= 180] ( 0.5,-0.5) [line width=1pt];
\draw [red] (-0.5,-1) to[out= 135, in= -45] (-1.3,-0.2) [line width=1pt];
\fill [red] (-0.5,-1) circle (2.5pt); \fill [white] (-0.5,-1) circle (1.8pt);
\draw (-1.3,-0.2) node[left]{$\ta_{\rbullet}$};
% permissible curve
\draw[blue][line width=1pt] (-1,-1) to[out=45,in=180] (2,0);
\draw[blue] (2,0) node[below]{$c$};
% subcaption
\draw (0,-1.5) node{Case $\bfII$};
\end{tikzpicture}
\ \
\begin{tikzpicture}
% boundary
\draw[black] (-1,-1) node[below]{$\bSurf$};
\draw[black][->][line width=1.2pt] (2,1) -- (0,1) to[out=180,in=90] (-1,0) to[out=-90,in=180] (0,-1) -- (2,-1);
% FFASgreen
\draw[qqwuqq] ( 0,-1)--( 0, 1) [line width=1pt];
\fill[qqwuqq] ( 0,-1) circle (0.1cm);
\fill[qqwuqq] ( 0, 1) circle (0.1cm);
% FFASred
\draw [red] (-1,0) -- (2,0) [line width=1pt];
\fill [red] (-1,0) circle (2.5pt); \fill [white] (-1,0) circle (1.8pt);
% permissible curve
\draw[blue][line width=1pt] (-1,0) to[out=45,in=180] (2,-0.2);
\draw[blue] (1.9,-0.6) node[above]{$c$};
% subcaption
\draw (0.5,-1.5) node{Case $\bfIII$};
\end{tikzpicture}
\caption{The end segments of permissible curves corresponding to string modules. }
\label{fig:endsegment}
\end{figure}

Denote by $c^{\circlearrowleft}$ the curve obtained by the following steps:
\begin{itemize}
  \item[Step 1] If $c_{(0,1)}$ (resp., $c_{(m(c),m(c)+1)}$) is of the form of Case $\bfII$ or Case $\bfIII$,
    then move $c(0)$ (resp., $c(1)$) to the next $\gbullet$-marked point along the positive direction of $\bSurf$;
    otherwise, fix $c_{(0,1)}$ (resp., $c_{(m(c),m(c)+1)}$).
  \item[Step 2] Fixing all segments $c_{(i,i+1)}$ of $c$, where $1\le i\le m(c)-1$.
\end{itemize}
Then, we have

\begin{proposition}\label{prop:curve change}
Let $M$ be a string module in $\modcat A$ and $\bfP(M)$ be the projective representation of $M$. Let $c_M$ be the permissible curve in $\PC(\SURF^{\F_A}(A))$ such that $\MM(c_M)\cong M$ and $\tc_{\bfP(M)}$ be the admissible curve in $\AC(\SURF^{\F_A}(A))$ such that $\X(\tc_{\bfP(M)})\cong \bfP(M)$. Then $c_{\bfP(M)}$ is homotopic to $c_M^{\circlearrowleft}$ up to fixing endpoints.
\end{proposition}

\begin{proof} We only prove the above three cases for the end segment $c_{(0,1)}$ of $c_M$. The proof of the end segment $c_{(m(c_M),m(c_M)+1)}$ of $c_M$ is similar.

(1) If the end segment $c_{(0,1)}$ of $c_M$ is of the form of the Case $\bfI$, then the string corresponding to $M$ is the shadow part shown in the following
\begin{center}
\begin{tikzpicture}
\draw (0,0) node{\xymatrix{
\mathfrak{v}(a_{\gbullet}) & \mathfrak{v}(\agreen{c_M}{1}) \ar[l] \ar@{-}[r] & \cdots}};
\draw[pattern color=orange, pattern=north west lines][opacity=0.5]
(-0.5,-0.5) -- (-0.5,0.5) -- (2.5, 0.5) -- (2.5,-0.5) -- cycle;
\end{tikzpicture}
\end{center}
In this case, we have $\overleftarrow{\agreen{c_M}{1}}=a_{\gbullet}$.
Thus, by Corollary \ref{coro:proj resolution}, the projective representation of $M$ satisfies that the indecomposable projective module $\MM(c_{\proj}(\overleftarrow{\agreen{c_M}{1}}))$ is isomorphic to $P(\mathfrak{v}(a_{\gbullet}))$.
Then, we can draw the curves $c_M$ and $c_{\bfP(M)}$ as follows, where the end segment $\tc_{[0,1]}$ of $\tc_{\bfP(M)}$ crosses the graded $\rbullet$-arc $\ta_{\rbullet}$ which intersects with $a_{\gbullet}$.

\begin{figure}[H]
\definecolor{ffqqqq}{rgb}{1,0,0}
\definecolor{qqwuqq}{rgb}{0,0.5,0}
\begin{tikzpicture}[scale=1.5]
% boundary
\draw[black] (-2,-1) node[below]{$\bSurf$};
\draw[black][<-][line width=1.2pt] (-2, 1) -- ( 2, 1);
\draw[black][->][line width=1.2pt] (-2,-1) -- ( 2,-1);
% FFASgreen
\draw[qqwuqq] ( 0,-1)--( 0, 1) [line width=1pt];
\fill[qqwuqq] ( 0,-1) circle (0.1*0.66cm);
\fill[qqwuqq] ( 0, 1) circle (0.1*0.66cm);
\draw[qqwuqq] ( 0,-1)--(-1, 1) [line width=1pt];
\fill[qqwuqq] (-1, 1) circle (0.1*0.66cm);
\draw (0,-0.25) node[right]{$\agreen{c_M}{1}$};
\draw (-0.25,-0.25) node[left]{$a_{\gbullet}=\overleftarrow{\agreen{c_M}{1}}$};
% FFASred
\draw [red] (-0.5,1) to[out= -45, in= 180] ( 0.3, 0.50) [line width=1pt];
\draw [red] ( 0.3, 0.50) -- ( 0.5, 0.5) to [out=0,in=-135] (1.5,1) [line width=1pt][dash pattern=on 2pt off 2pt];
\draw [red] (1.5,1) -- (1.2,-1) [line width=1pt];
\draw [red] (-0.5,1) to[out=-135, in=  45] (-1.30, 0.20) [line width=1pt];
\fill [red] (-0.50, 1.00) circle (2.5*0.66pt); \fill [white] (-0.50, 1.00) circle (1.8*0.66pt);
\fill [red] ( 1.50, 1.00) circle (2.5*0.66pt); \fill [white] ( 1.50, 1.00) circle (1.8*0.66pt);
\fill [red] ( 1.20,-1.00) circle (2.5*0.66pt); \fill [white] ( 1.20,-1.00) circle (1.8*0.66pt);
\draw (-1.3, 0.2) node[left]{$\ta_{\rbullet} = \ared{\tc_{\bfP(M)}}{1}$};
% permissible curve
\draw[blue][line width=1pt][->] (-1,1) to[out=-45,in=180] (2,0);
\draw[blue] (2,0) node[above]{$c_M$};
% admissible curve
\draw[orange][line width=1pt][->] (-1,1) to[out=-50,in=180] (2,-0.2);
\draw[orange] (2,-0.2) node[below]{$\tc_{\bfP(M)}$};
\end{tikzpicture}
\caption{Case $\mathbf{I}$}
\label{fig:pf in prop:curve change I}
\end{figure}

(2) If the end segment $c_{(0,1)}$ of $c_M$ is of the form of the Case $\bfII$, then the string corresponding to $M$ is the shadow part shown in the following:
\begin{center}
\begin{tikzpicture}
\draw (0,0) node{\xymatrix{
\mathfrak{v}(a_{\gbullet}) \ar[r] & \mathfrak{v}(\agreen{c_M}{1}) \ar@{-}[r] & \cdots}};
\draw[pattern color=orange, pattern=north west lines][opacity=0.5]
(-0.5,-0.5) -- (-0.5,0.5) -- (2.5, 0.5) -- (2.5,-0.5) -- cycle;
\end{tikzpicture}
\end{center}
In this case, the indecomposable projective module $\MM(c_{\proj}(\overleftarrow{\agreen{c_M}{1}}))$ is zero. Then, by Corollary \ref{coro:proj resolution}, we can draw the curves $c_M$ and $c_{\bfP(M)}$ as follows (we have two cases (2.1) and (2.2)), where the end segment $\tc_{[0,1]}$ of $\tc_{\bfP(M)}$ crosses no $\ta_{\rbullet}$ which intersects with $a_{\gbullet}$, but it crosses the graded $\rbullet$-arc which intersects with $\agreen{c_M}{1}$.
\begin{figure}[htbp]
\definecolor{ffqqqq}{rgb}{1,0,0}
\definecolor{qqwuqq}{rgb}{0,0.5,0}
\begin{tikzpicture}[scale=1.5]
% boundary
\draw[black] (-2,-1) node[below]{$\bSurf$};
\draw[black][<-][line width=1.2pt] (-2, 1) -- ( 2, 1);
\draw[black][->][line width=1.2pt] (-2,-1) -- ( 2,-1);
% FFASgreen
\draw[qqwuqq] ( 0,-1)--( 0, 1) [line width=1pt];
\fill[qqwuqq] ( 0,-1) circle (0.1*0.66cm);
\fill[qqwuqq] ( 0, 1) circle (0.1*0.66cm);
\draw[qqwuqq] ( 0, 1)--(-1,-1) [line width=1pt];
\fill[qqwuqq] (-1,-1) circle (0.1*0.66cm);
\draw[qqwuqq] ( 0, 1)--( 1,-1) [line width=1pt];
\fill[qqwuqq] ( 1,-1) circle (0.1*0.66cm);
\draw (-0.3, 0.25) node[left]{$a_{\gbullet}$};
\draw ( 0  , 0.25) node{$\agreen{c_M}{1}$};
\draw ( 0.3, 0.25) node[right]{$\agreen{c_M}{2}$};
% FFASred
\draw [red] (-0.5,-1) to[out=  45, in= 135] ( 0.5,-1) [line width=1pt];
\draw [red] (-0.5,-1) to[out= 135, in= -45] (-1.3,-0.2) [line width=1pt];
\draw [red] ( 0.5,-1) -- (1.5,1) [line width=1pt];
\fill [red] (-0.5,-1) circle (2.5*0.66pt); \fill [white] (-0.5,-1) circle (1.8*0.66pt);
\fill [red] ( 0.5,-1) circle (2.5*0.66pt); \fill [white] ( 0.5,-1) circle (1.8*0.66pt);
\fill [red] ( 1.5, 1) circle (2.5*0.66pt); \fill [white] ( 1.5, 1) circle (1.8*0.66pt);
\draw (-1.3,-0.2) node[left]{$\ta_{\rbullet}$};
\draw [red][->] (-0.25,-0.85) -- (-0.25,-0.7) -- (-1.3,-0.7);
\draw (-1.3,-0.7) node[left]{$\ared{\tc_{\bfP(M)}}{1}$};
% permissible curve
\draw[blue][->][line width=1pt] (-1,-1) to[out=45,in=180] (2,0);
\draw[blue] (2,0) node[above]{$c_M$};
\draw[orange][->][line width=1pt] (0,-1) to[out=45,in=180] (2,-0.2);
\draw[orange] (2,-0.2) node[below]{$\tc_{\bfP(M)}$};
% subcaption
\draw (0,-1.5) node{(2.1)};
\end{tikzpicture}
\ \
\begin{tikzpicture}[scale=1.5]
% boundary
\draw[black] (-2,-1) node[below]{$\bSurf$};
\draw[black][<-][line width=1.2pt] (-2, 1) -- ( 2, 1);
\draw[black][->][line width=1.2pt] (-2,-1) -- ( 2,-1);
% FFASgreen
\draw[qqwuqq] ( 0,-1)--( 0, 1) [line width=1pt];
\fill[qqwuqq] ( 0,-1) circle (0.1*0.66cm);
\fill[qqwuqq] ( 0, 1) circle (0.1*0.66cm);
\draw[qqwuqq] ( 0, 1)--(-1,-1) [line width=1pt];
\fill[qqwuqq] (-1,-1) circle (0.1*0.66cm);
\draw[qqwuqq] ( 0,-1)--( 1, 1) [line width=1pt];
\fill[qqwuqq] ( 1, 1) circle (0.1*0.66cm);
\draw (-0.3, 0.25) node[left]{$a_{\gbullet}$};
\draw ( 0  , 0.25) node{$\agreen{c_M}{1}$};
\draw ( 0.6, 0.25) node[right]{$\agreen{c_M}{2}$};
% FFASred
\draw [red] (-0.5,-1) to[out= 135, in= -45] (-1.3,-0.2) [line width=1pt];
\draw [red] (-0.5,-1) -- ( 0.5, 1) [line width=1pt];
\draw [red] ( 0.5, 1) to[out=-45,in=-180] ( 1,0.7) [line width=1pt];
\draw [red] ( 1, 0.7) to[out=  0,in=-135] (1.5, 1) [line width=1pt][dash pattern=on 2pt off 2pt];
\draw [red] ( 1.5, 1) -- ( 1.3,-1) [line width=1pt];
\fill [red] (-0.5,-1) circle (2.5*0.66pt); \fill [white] (-0.5,-1) circle (1.8*0.66pt);
\fill [red] ( 0.5, 1) circle (2.5*0.66pt); \fill [white] ( 0.5, 1) circle (1.8*0.66pt);
\fill [red] ( 1.5, 1) circle (2.5*0.66pt); \fill [white] ( 1.5, 1) circle (1.8*0.66pt);
\fill [red] ( 1.3,-1) circle (2.5*0.66pt); \fill [white] ( 1.3,-1) circle (1.8*0.66pt);
\draw (-1.3,-0.2) node[left]{$\ta_{\rbullet}$};
% permissible curve
\draw[blue][->][line width=1pt] (-1,-1) to[out=45,in=180] (2,0);
\draw[blue] (2,0) node[above]{$c_M$};
\draw[orange][->][line width=1pt] (0,-1) to[out=45,in=180] (2,-0.2);
\draw[orange] (2,-0.2) node[below]{$\tc_{\bfP(M)}$};
% subcaption
\draw (0,-1.5) node{(2.2)};
\end{tikzpicture}
\caption{Case $\mathbf{II}$ }
\label{fig:pf in prop:curve change II}
\end{figure}

(3) The proof of Case $\bfIII$ is similar to (2).
\end{proof}

In fact, we have shown the following result.

\begin{theorem}\label{thm:curve embeding}
There is an injection
\[\cemb: \PC(\SURF^{\F_A}(A)) \to \AC(\SURF^{\F_A}(A)),\ c\mapsto \tc^{\circlearrowleft}\]
from $\PC(\SURF^{\F_A}(A))$ to $\AC(\SURF^{\F_A}(A))$ such that $\X(\tc^{\circlearrowleft}) \cong \bfP(\MM(c))$ in $\per A$.
\end{theorem}

\section{The proof of Theorem \ref{thm2}.}\label{sect:main result}

In this section, we prove that there are no strictly shod algebras in hereditary gentle algebras. First, we recall some definitions.

\begin{definition} \rm
Let $A$ be a finite dimensional algebra and let $M$ be an $A$-module and $P$ a projective $A$-module.
\begin{itemize}
\item[(1)] $M$ is called {\it $\tau$-rigid} if $\Hom_{A}(M,\tau M)=0$ and
$M$ is called {\it $\tau$-tilting} if $M$ is $\tau$-rigid and $|M|=|A|$.
\item[(2)] The pair $(M,P)$ is called {\it $\tau$-rigid} if $M$ is $\tau$-rigid and $\Hom_{A}(P, M) = 0$.
\item[(3)] The pair $(M,P)$ is called {\it support $\tau$-tilting} if $(M,P)$ is $\tau$-rigid and $|M|+|P|=|A|$. In this case, $M$ is a support $\tau$-tilting module.
\end{itemize}
\end{definition}

Denote by $\mathrm{s\tau}$-$\mathrm{tilt}A$ the set of isomorphism classes of basic support $\tau$-tilting $A$-modules.

\begin{definition} \rm
\label{def:silting-objects}
Let $A$ be a finite dimensional algebra and $P$ be a complex in $K^{b}(\proj A)$.
 \begin{itemize}
\item[(1)] $P$ is called {\it presilting} if $\Hom_{K^{b}(\proj A)}(P,P[i])=0$ for $i>0$.
\item[(2)] $P$ is called {\it silting} if it is presilting and generates $K^{b}(\proj A)$ as a triangulated category.
\item[(3)] $P$ is called {\it 2-term} if it only has non-zero terms in degrees $-1$ and $0$.
\end{itemize}
\end{definition}

We denote by $\mathrm{2}$-$\mathrm{silt}A$ the set of isomorphism classes of basic 2-term silting complexes in $K^{b}(\proj A)$. The following \cite[Theorem 3.2]{AIR2014} gives the construction of 2-term silting complexes by support $\tau$-tilting modules.

\begin{theorem}\label{thm:bijection-between-silting-objects-and-support-tilting-modules}
There exists a bijection
\begin{center}
$\mathrm{2}$-$\mathrm{silt}A \leftrightarrow \mathrm{s\tau}$-$\mathrm{tilt}A$
\end{center}
given by $\mathrm{2}$-$\mathrm{silt}A\ni P \mapsto H^{0}(P) \in \mathrm{s}\tau$-$\mathrm{tilt}A$ and $\mathrm{s}\tau$-$\mathrm{tilt}A\ni (M,P) \mapsto(P_{1}\oplus P \xrightarrow {(f\ 0)} P_{0})\in\mathrm{2}$-$\mathrm{silt}A$, where $f: P_{1} \rightarrow P_{0}$ is a minimal projective
presentation of $M$.
\end{theorem}

\begin{definition}\rm
\label{def:silted-algebras}
Let $Q$ be an acyclic quiver and $A=\kk Q$. We call an algebra $B$ {\it silted} of type $Q$ if there exists a 2-term silting complex $P$ over $A$ such that $B=\End_{K^{b}(\proj A)} (P)$.
\end{definition}

Next we recall the definition of strictly shod algebras in \cite{CL1999}.

\begin{definition}\rm
\label{def:strictly-shod-algebras}
An algebra $A$ is called {\it shod} (for small homological dimension) if for each indecomposable $A$-module $X$, either the projective or the injective dimension is at most one. $A$ is called {\it strictly shod} if it is shod and $\mathrm{gl.dim} A=3$.
\end{definition}

Note that tilted algebras are silted. Moreover, any silted algebras is shod. Thus, by the definition of strictly shod algebras, we only need to show that there are no silted algebras of global dimension 3 in hereditary gentle algebras. Recall that we have described the silted algebras of gentle algebras in \cite{LZ2022}. Let $A$ be a gentle algebra and $S$ be a complex of $\per A$. Then $S$ can induce a graded algebra $A^{S}$. In particular, if $S$ is a silting complex, we have
%We say that $C$ is a {\it complex without intersection} if $D_{\gbullet}(C):=\{\tgamma_i\mid i\in I\}$ is a dissection of $\SURFred{\F_A}(A)$ such that $\tgamma_i\cap\tgamma_j\cap(\innerSurf)=\varnothing$. In this case, the complex $C$ can induce a graded algebra $A^C=\kk\Q^C/\I^C$ by graded marked ribbon surface as following:
%
%Step 1: The quiver $\Q^C$ of $A^C$ is obtained by:
%  \begin{itemize}
%    \item The vertex set of $\Q^C$, say $\Q^C_0$, is the set $D_{\gbullet}(C):=\{\tgamma_i \mid i\in I\}$.
%
%    \item  If $\gamma_i$ and $\gamma_j$ has an intersection $p\in \M_A$ such that $\gamma_i$ is left to $\gamma_j$ at $p$,
%      and, for all $\tgamma_l\in D$ with endpoint $p$, $\tgamma_l$ is not between $\tgamma_i$ and $\tgamma_j$, then there exists an arrow $\alpha$ from $\tgamma_j$ to $\tgamma_i$ in the arrow set $\Q^C_1$ of $\Q^C$.
%      We define $|\alpha| = \ii_p(\tgamma_i, \tgamma_j)$.
%  \end{itemize}
%
%Step 2: For two arrows $\alpha: \gamma_i\to \gamma_j$ and $\beta: \gamma_j \to \gamma_k$,
%    $\alpha\beta\in \I^C$ if and only if $\gamma_i, \gamma_j, \gamma_k$ are edges of the same polygon obtained by $D_{\gbullet}(C)$ cutting $\SURFred{\F_A}$.
\begin{theorem}\label{thm-endo-alg}
Let $A$ be a gentle algebra and $S\in \per(A)$ a silting complex. Then $\H^0(A^S) \cong \End_{\per A}(S)$.
\end{theorem}

By Theorem \ref{thm-endo-alg}, we know that the silted algebras of hereditary gentle algebra $A$ can be viewed as the 0-th cohomology of the graded algebra $A^{S}$ which induced by the $2$-term silting complex $S$. Thus, we need to show that there are no algebras $\H^0(A^S)$ of global dimension 3. Thanks to the work of the second author of this paper and his coauthors \cite{LGH2022}, they found that the global dimension of $A$ can be calculated by the edges of the polygons which obtained by the full formal arc system cutting the marked surface of $A$.

To be precise,  let $\{\Delta_i\mid 1\le i\le d\}$ be the set of elementary polygons which obtained by $\Dgreen$ cutting $\Surf(A)$. Denote by $\mathfrak{C}(\Delta_i)$ the number of edges of $\Delta_i$ belong to $\Dgreen$ if the unmarked boundary component of $\Surf(A)$ is not an edge of $\Delta_i$; otherwise, $\mathfrak{C}(\Delta_i)=\infty$. Then, we have

\begin{theorem}\label{thm:global-dimension}
The global dimension of $A$ is
\begin{center}
$\gldim A=\max\limits_{1\leq i\leq d}{\mathfrak{C}(\Delta_i)}-1$.
\end{center}
\end{theorem}

\begin{proposition} \label{prop:FFAS}
Let $\Gamma = \{ c_i \in \PC(\SURF^{\F_A}(A)) | 1\le i\le n\}$ be a triangulation of the marked ribbon surface $\SURF^{\F_A}(A)$ of $A$. Then the following statements hold.
\begin{itemize}
  \item[\rm(1)] $\MM(\Gamma):=\bigoplus_{1\le i\le n}\MM(c_i)$ is a support $\tau$-tilting $A$-module.
  \item[\rm(2)] $\Gamma^{\circlearrowleft} := \{c_i^{\circlearrowleft} \mid \MM(c_i)\neq0\}\bigcup\{c_i \mid \MM(c_i)=0\}$
      is an $\gbullet$-$\mathrm{FFAS}$ of $\SURF^{\F_A}(A)$. Furthermore, $\X(\tGamma^{\circlearrowleft}):=\bigoplus_{1\le i\le n}\X(\widehat{c_i})$ is a $2$-term silting object in $\per A$, where
      \[\widehat{c_i}=\left\{\begin{array}{ll}
      \tc_i^{\circlearrowleft}&\mathrm{if}\ \MM(c_i)\neq0,\\
      \tc_i&\mathrm{if}\ \MM(c_i)=0,
\end{array}\right.\]
    and $\tc_i^{\circlearrowleft}$ $($resp., $\tc_i$$)$ is the graded curve of $c_i^{\circlearrowleft}$ $($resp., $c_i$$)$ with some grading.
\end{itemize}
\end{proposition}

\begin{proof}
(1) By \cite[Theorem B]{HZZ2020}, we have a bijection
\[\Tri(\SURF^{\F_A}(A)) \to \stautilt(A), \Gamma\mapsto \MM(\Gamma)\]
between the set $\Tri(\SURF^{\F_A}(A))$ of all triangulations of $\SURF^{\F_A}(A)$
and the set $\stautilt(A)$ of all isoclasses of support $\tau$-tilting $A$-modules.

(2) For any $\Gamma = \{c_i | 1\le i\le n\}\in\Tri(\SURF^{\F_A}(A))$, by Proposition \ref{prop:curve change} and Theorem \ref{thm:bijection-between-silting-objects-and-support-tilting-modules}, we have
\[\X(\tGamma^{\circlearrowleft}) = \bigoplus_{1\le i\le n}\X(\widehat{c_i})
\cong \bigoplus_{1\le i\le n \atop \MM(c_i)\ne 0}\X(\tc_{\bfP(\MM(c_i))})
\oplus \bigoplus_{1\le j\le n \atop \MM(c_j) = 0} \X(\tc_j), \]
where % $\tc_i^{\circlearrowleft} \simeq \tc_{\bfP(\MM(c_i))}$ is the admissible curve corresponding to the complex in $\per A$ induced by the projective representation of $\MM(c_i)$, and
$c_j$ $( = a_{\gbullet} \in \Dgreen(A))$ is the admissible curve with grading $\tc_j$ such that
$\X(\tc_j)[-1]$ is the shift of indecomposable projective module corresponding to $a_{\gbullet}$.
Note that any two admissible curves in $\Gamma^{\circlearrowleft}$ has no intersection in $\innerSurf$.
Otherwise, there are $x$ and $y$ ($1\le x\ne y\le n$) such that $\hat{c}_x \cap \hat{c}_y \cap (\innerSurf) \ne \varnothing$. Then, we have
\begin{center}
  $\dim_{\kk}\Hom_{\per A}(\X(\hat{c}_x), \X(\hat{c}_y)[d]) \ne 0$ \\
  and $\dim_{\kk}\Hom_{\per A}(\X(\hat{c}_y), \X(\hat{c}_x)[-d]) \ne 0$,
\end{center}
where $d = \ii_p(\hat{c}_x, \hat{c}_y)$ and $p\in \hat{c}_x \cap \hat{c}_y \cap (\innerSurf)$.
This shows that $\X(\tGamma^{\circlearrowleft})$ is not silting, we obtain a contradiction. Thus, by \cite[Proposition 1.11]{APS2019}, we obtain that $\Gamma^{\circlearrowleft}$ is an $\gbullet$-FFAS.
\end{proof}

\begin{theorem}\label{main thm}
Let $A$ be a hereditary gentle algebra. Then there are no strictly shod algebras in $A$.
\end{theorem}

\begin{proof} Let $S$ be a $2$-term silting complex in $\per A$. By Theorem \ref{thm:bijection-between-silting-objects-and-support-tilting-modules}, $\mathrm{H}^0(S)$ is a support $\tau$-tilting module over $A$.
Then there is a triangulation $\Gamma$ of $\SURF^{\F_A}(A)$ such that $\mathrm{H}^0(S)\cong \bigoplus_{c\in\Gamma}\MM(c)$. By Proposition \ref{prop:FFAS} and Theorem \ref{thm:curve embeding}, $\Gamma^{\circlearrowleft}$ is an $\gbullet$-FFAS. Then
\[B = \mathrm{H}^0(\grEnd_{\per A}(S)) \cong \mathrm{H}^0(\grEnd_{\per A}(\X(\Gamma^{\circlearrowleft}))) \]
where $\grEnd_{\per A}(S)$ is the dg endomorphism algebra of $S$. Let $\Tri(\Gamma)$ be the set of all triangles obtained by $\Gamma$ cutting $\SURF^{\F_A}(A)$ and $\Delta(\Gamma^{\circlearrowleft})$ be the set of all elementary $\gbullet$-polygons obtained by $\Gamma^{\circlearrowleft}$ cutting $\SURF^{\F_A}(A)$.
Next, we show that the number of all edges of any elementary $\gbullet$-polygon in $\Delta(\Gamma^{\circlearrowleft})$ is at most $4$.

For convenience, we always use $c^{\circlearrowleft}$ to represent the curves in $\Gamma^{\circlearrowleft}$. Now consider all vertices of triangles in $\Tri(\Gamma)$. We have two cases in \Pic \ref{fig:main thm 1} ($t\in\NN^+$).
\begin{figure}[H]
\definecolor{ffqqqq}{rgb}{1,0,0}
\definecolor{qqwuqq}{rgb}{0,0.5,0}
\begin{tikzpicture}
% boundary
\draw[black] (-2.5,0) node[below]{$\bSurf$};
\draw[black][->][line width=1.2pt] (-2.5,0) -- ( 2.5,0);
% FFASgreen
\fill[qqwuqq] (-1.5, 0) circle (0.1cm);
\fill[qqwuqq] ( 1.5, 0) circle (0.1cm);
\draw[qqwuqq] ( 1.5, 0) arc (0:180:1.5) [line width = 1pt];
% FFASred
\fill [red] (0,0) circle (2.5pt); \fill [white] (0,0) circle (1.8pt);
% permissible curve
\draw[blue] ( 0.00, 0.00) -- (-1.41, 1.41) [line width=1pt][->];
\draw[blue] (-1.41, 1.41) node[above]{$c_1$};
\draw[blue] ( 0.00, 0.00) -- ( 0.00, 2.00) [line width=1pt][->];
\draw[blue] (-0.00, 2.00) node[above]{$c_i$};
\draw[blue] ( 0.00, 0.00) -- ( 1.41, 1.41) [line width=1pt][->];
\draw[blue] ( 1.41, 1.41) node[above]{$c_t$};
% \draw[blue][dotted] (0,0.8) arc(90:45:0.8) [line width=1pt];
\draw[blue] (-0.30, 0.50) node[above]{$\cdots$};
\draw[blue] ( 0.30, 0.50) node[above]{$\cdots$};
% subcaption
\draw (0,-0.5) node{($1\leq i\leq t$)};
\draw (0,-1.0) node{(1)};
\end{tikzpicture}
\ \ \ \
\begin{tikzpicture}
% boundary
\draw[black] (-2.5,0) node[below]{$\bSurf$};
\draw[black][->][line width=1.2pt] (-2.5,0) -- ( 2.5,0);
% FFASgreen
\fill[qqwuqq] (-2, 0) circle (0.1cm);
\fill[qqwuqq] ( 0, 0) circle (0.1cm);
\fill[qqwuqq] ( 2, 0) circle (0.1cm);
\draw[qqwuqq][rotate around={ 25:(0,0)}] (1.5,0) arc (0:20:1.5) [line width = 1pt];
\draw[qqwuqq][rotate around={ 60:(0,0)}] (1.5,0) arc (0:20:1.5) [line width = 1pt];
\draw[qqwuqq][rotate around={ 95:(0,0)}] (1.5,0) arc (0:20:1.5) [line width = 1pt];
\draw[qqwuqq][rotate around={135:(0,0)}] (1.5,0) arc (0:20:1.5) [line width = 1pt];
% FFASred
\fill [red] (-1,0) circle (2.5pt); \fill [white] (-1,0) circle (1.8pt);
\fill [red] ( 1,0) circle (2.5pt); \fill [white] ( 1,0) circle (1.8pt);
% permissible curve
\draw[blue][rotate around={ 36:(0,0)}] ( 0.00, 0.00) -- ( 2.00, 0.00) [line width=1pt][->];
\draw[blue][rotate around={ 72:(0,0)}] ( 0.00, 0.00) -- ( 2.00, 0.00) [line width=1pt][->];
\draw[blue][rotate around={108:(0,0)}] ( 0.00, 0.00) -- ( 2.00, 0.00) [line width=1pt][->];
\draw[blue][rotate around={144:(0,0)}] ( 0.00, 0.00) -- ( 2.00, 0.00) [line width=1pt][->];
\draw[blue][rotate around={ 36:(0,0)}] ( 2.20, 0.00) node{$c_t$};
\draw[blue][rotate around={ 72:(0,0)}] ( 2.20, 0.00) node{$c_{i}$};
\draw[blue][rotate around={108:(0,0)}] ( 2.20, 0.00) node{$c_{i-1}$};
\draw[blue][rotate around={144:(0,0)}] ( 2.20, 0.00) node{$c_1$};
% \draw[blue][dotted] (0,0.8) arc(90:45:0.8) [line width=1pt];
\draw[blue] (-0.60, 0.50) node[above]{$\cdots$};
\draw[blue] ( 0.60, 0.50) node[above]{$\cdots$};
% subcaption
\draw (0,-0.5) node{($1\leq i\leq t$)};
\draw (0,-1.0) node{(2)};
\end{tikzpicture}
\caption{Two cases of $\Tri(\Gamma)$}
\label{fig:main thm 1}
\end{figure}
\noindent
\begin{itemize}
\item[Case] (1) We have $c_1(0)=\cdots=c_t(0) = p$ is an extra marked point.
  Thus, by Proposition \ref{prop:curve change}, we obtain that
  \[c^{\circlearrowleft}_1(0)=\cdots=c^{\circlearrowleft}_t(0) \ne p \]
  and $c^{\circlearrowleft}_1$, $\ldots$, $c^{\circlearrowleft}_t$ are of the form shown in
  \Pic \ref{fig:main thm 2} (1). In this case, all triangles in $\Tri(\Gamma)$ with vertex $p$ changes to
  elementary $\gbullet$-polygons in $\Delta(\Gamma)$ with vertices $p^{\circlearrowleft}$.

\item[Case] (2) If there exists an integer $1\le i\le t$ such that $c_i(0) \ne c^{\circlearrowleft}_i(0)$
  (resp., $c_i(0) = c^{\circlearrowleft}_i(0)$),
  then $c_j(0) \ne c^{\circlearrowleft}_j(0)$ (resp., $c_j(0) = c^{\circlearrowleft}_j(0)$) for any $i\le j\le t$ (resp., $1\le j\le i$).
  That is, $c^{\circlearrowleft}_1$, $\ldots$, $c^{\circlearrowleft}_t$ are of the form shown in
  \Pic \ref{fig:main thm 2} (2).
  Otherwise, we have $c_i(0)\ne c^{\circlearrowleft}_i(0)$ (resp., $c_i(0)= c^{\circlearrowleft}_i(0)$)
  and $c_j(0)=c^{\circlearrowleft}_j(0)$ (resp., $c_j(0)\ne c^{\circlearrowleft}_j(0)$) for some $j>i$ (resp., $j<i$),
  then $c^{\circlearrowleft}_i \cap c^{\circlearrowleft}_j \cap (\innerSurf) \ne \varnothing$,
  this is a contradiction by Proposition \ref{prop:FFAS} (2).

  Therefore, there exists a unique $1\le \imath\le t$ satisfying $c_{\imath}(0) \ne c^{\circlearrowleft}_{\imath}(0)$ such that
  $c_{\hbar}(0) = c^{\circlearrowleft}_{\hbar}(0)$ for all $1\le \hbar < \imath$
  and $c_{\jmath}(0) \neq c^{\circlearrowleft}_{\jmath}(0)$ for all $\imath < \jmath \le t$.
  Then, by Proposition \ref{prop:FFAS} (2), $c^{\circlearrowleft}_{\imath-1}(1) = c^{\circlearrowleft}_{\imath}(1)$. Thus,
  the triangle, say $\mathcal{T}$, with edges $c_{\imath-1}$ and $c_{\imath}$ changes to the elementary $\gbullet$-polygon, say $\mathcal{P}$, with edges $c_{\imath-1}^{\circlearrowleft}$, $c_{\imath}^{\circlearrowleft}$ and $pp^{\circlearrowleft}$.
\end{itemize}

\begin{figure}[htbp]
\definecolor{ffqqqq}{rgb}{1,0,0}
\definecolor{qqwuqq}{rgb}{0,0.5,0}
\begin{tikzpicture}
% boundary
\draw[black] (-2.5,0) node[below]{$\bSurf$};
\draw[black][->][line width=1.2pt] (-2.5,0) -- ( 2.5,0);
% FFASgreen
\fill[qqwuqq] (-1.5, 0) circle (0.1cm);
\fill[qqwuqq] ( 1.5, 0) circle (0.1cm);
\draw[qqwuqq] ( 1.5, 0) arc (0:180:1.5) [line width = 1pt];
% FFASred
\fill [red] (0,0) circle (2.5pt); \fill [white] (0,0) circle (1.8pt);
% permissible curves
\draw[blue] ( 0.00, 0.00) -- (-1.41, 1.41) [line width=1pt][->];
% \draw[blue] (-1.41, 1.41) node[above]{$c_1$};
\draw[blue] ( 0.00, 0.00) -- ( 0.00, 2.00) [line width=1pt][->];
% \draw[blue] (-0.00, 2.00) node[above]{$c_i$};
\draw[blue] ( 0.00, 0.00) -- ( 1.41, 1.41) [line width=1pt][->];
% \draw[blue] ( 1.41, 1.41) node[above]{$c_t$};
% admissible curves
\draw[orange] (-1.41, 1.41) -- (-1.06, 1.06) to[out=-45,in=170] ( 1.50, 0.00) [line width=1pt][<-];
\draw[orange] (-1.41, 1.41) node[above]{$\tc_1^{\circlearrowleft}$};
\draw[orange] (-0.  , 2.  ) -- ( 0.  , 1.50) to[out=-90,in=160] ( 1.50, 0.00) [line width=1pt][<-];
\draw[orange] (-0.00, 2.00) node[above]{$\tc_i^{\circlearrowleft}$};
\draw[orange] ( 1.41, 1.41) -- ( 1.06, 1.06) to[out=225,in=135] ( 1.50, 0.00) [line width=1pt][<-];
\draw[orange] ( 1.41, 1.41) node[above]{$\tc_t^{\circlearrowleft}$};
% \draw[blue][dotted] (0,0.8) arc(90:45:0.8) [line width=1pt];
\draw[blue] (-0.30, 0.50) node[above]{$\cdots$};
\draw[blue] ( 0.30, 0.50) node[above]{$\cdots$};
% subcaption
\draw (0,0) node[below]{$p$};
\draw (1.5,0) node[below]{$p^{\circlearrowleft}$};
\draw (0,-1.0) node{(1)};
\end{tikzpicture}
\ \ \ \
\begin{tikzpicture}
% boundary
\draw[black] (-2.5,0) node[below]{$\bSurf$};
\draw[black][->][line width=1.2pt] (-2.5,0) -- ( 2.5,0);
% FFASgreen
\fill[qqwuqq] (-2, 0) circle (0.1cm);
\fill[qqwuqq] ( 0, 0) circle (0.1cm);
\fill[qqwuqq] ( 2, 0) circle (0.1cm);
\draw[qqwuqq][rotate around={ 25:(0,0)}] (1.5,0) arc (0:20:1.5) [line width = 1pt];
\draw[qqwuqq][rotate around={ 60:(0,0)}] (1.5,0) arc (0:20:1.5) [line width = 1pt];
\draw[qqwuqq][rotate around={ 95:(0,0)}] (1.5,0) arc (0:20:1.5) [line width = 1pt];
\draw[qqwuqq][rotate around={135:(0,0)}] (1.5,0) arc (0:20:1.5) [line width = 1pt];
% FFASred
\fill [red] (-1,0) circle (2.5pt); \fill [white] (-1,0) circle (1.8pt);
\fill [red] ( 1,0) circle (2.5pt); \fill [white] ( 1,0) circle (1.8pt);
% permissible curve
\draw[blue][rotate around={ 36:(0,0)}] ( 0.00, 0.00) -- ( 2.00, 0.00) [line width=1pt][->];
\draw[blue][rotate around={ 72:(0,0)}] ( 0.00, 0.00) -- ( 2.00, 0.00) [line width=1pt][->];
\draw[blue][rotate around={108:(0,0)}] ( 0.00, 0.00) -- ( 2.00, 0.00) [line width=1pt][->];
\draw[blue][rotate around={144:(0,0)}] ( 0.00, 0.00) -- ( 2.00, 0.00) [line width=1pt][->];
% \draw[blue][dotted] (0,0.8) arc(90:45:0.8) [line width=1pt];
\draw[blue] (-0.60, 0.50) node[above]{$\cdots$};
\draw[blue] ( 0.60, 0.50) node[above]{$\cdots$};
\draw[orange] ( 1.62, 1.17) -- ( 1.21, 0.88) to[out=-144, in= 170] ( 2.00, 0.00) [line width=1pt][<-];
\draw[orange] ( 0.62, 1.90) -- ( 0.46, 1.42) to[out=-108, in= 170] ( 2.00, 0.00) [line width=1pt][<-];
\draw[orange][rotate around={108:(0,0)}] ( 0.00, 0.00) -- ( 2.00, 0.00) [line width=1pt][->];
\draw[orange][rotate around={144:(0,0)}] ( 0.00, 0.00) -- ( 2.00, 0.00) [line width=1pt][->];
\draw[orange][rotate around={ 36:(0,0)}] ( 2.20, 0.00) node{$c^{\circlearrowleft}_t$};
\draw[orange][rotate around={ 72:(0,0)}] ( 2.20, 0.00) node{$c^{\circlearrowleft}_{i}$};
\draw[orange][rotate around={108:(0,0)}] ( 2.20, 0.00) node{$c^{\circlearrowleft}_{i-1}$};
\draw[orange][rotate around={144:(0,0)}] ( 2.20, 0.00) node{$c^{\circlearrowleft}_1$};
% subcaption
\draw (2,0) node[below]{$p^{\circlearrowleft}$};
\draw (0,0) node[below]{$p$};
\draw (0,-1.0) node{(2)};
\end{tikzpicture}
\caption{Two cases of $\Gamma^{\circlearrowleft}$}
\label{fig:main thm 2}
\end{figure}

Now we have showed that the number of all edges of any elementary $\gbullet$-polygon in $\Delta(\Gamma^{\circlearrowleft})$ is less than or equal to $4$.

On the other hand, the quiver $\widehat{\Q}$ of $B=\kk\widehat{\Q}/\widehat{\I}$ can be obtained by deleting all arrows with non-zero grading of the graded quiver $\widetilde{\Q}$ of $\grEnd_{\per A}(S) = \kk\widetilde{\Q}/\I$.
Let $\Q$ be the quiver obtained by ignoring the gradings of all arrows of $\widetilde{\Q}$,
then $\widehat{\Q}$ is a subquiver of $\Q$ and $\widehat{\I}$ is generated by all paths of length $2$ lying in $\I$ such that the grading of each arrow is zero.
It is easy to see that
\[\gldim B = \gldim (\kk\widehat{\Q}/\widehat{\I}) \le \gldim (\kk\Q/\I).\]
By \cite[Theorem 1.1 (a)]{BZ2018}, we know that $\gldim (\kk\Q/\I)\leq 3$. Moreover, since the endomorphism algebras of silting objects are also gentle.
Thus, by Theorem \ref{thm:global-dimension}, we have
\begin{center}
$\gldim B\le \gldim (\kk\Q/\I)\le 3-1=2$.
\end{center}

%\begin{figure}[H]
%\definecolor{ffqqqq}{rgb}{1,0,0}
%\definecolor{qqwuqq}{rgb}{0,0.5,0}
%\begin{tikzpicture}[scale = 0.7]
%\draw (0,0) circle(2.0) [line width=1pt];
%\draw (0,0) circle(0.5) [line width=1pt];
%\draw[orange] (-1.41,-1.41) -- (1.41,-1.41) -- ( 1.41,1.41) -- (-1.41,1.41)  [line width=1pt];
%\draw[orange] (-1.41,1.41) -- (-1.41, -1.41) [dotted] [line width=1pt];
%\fill[qqwuqq] (-1.41,-1.41) circle (0.1cm);
%\fill[qqwuqq] ( 1.41,-1.41) circle (0.1cm);
%\fill[qqwuqq] ( 1.41, 1.41) circle (0.1cm);
%\fill[qqwuqq] (-1.41, 1.41) circle (0.1cm);
%\end{tikzpicture}
%\caption{ }
%\label{fig:main thm 3}
%\end{figure}
%This case show that $\SURF^{\F_A}(A)$ has a unmarked boundary component.
%Thus, $\gldim A=\infty$, i.e., $A$ is not hereditary, this is a contradiction.
%
%Therefore, we have $\gldim B\le \gldim \kk\Q/\I \le 3-1=2$.
\end{proof}

As a consequence of Theorem \ref{main thm}, we have

\begin{corollary}
Let $A$ be a hereditary algebra of Dynkin type $\mathbb{A}_{n}$ with arbitrary orientation. Then the silted algebras of $A$ have two forms:
\begin{itemize}
\item[{\rm(1)}] the tilted algebras of type $\mathbb{A}_{n}$;

\item[{\rm(2)}] the direct product of some tilted algebras of type $\mathbb{A}_{m_{1}},\ldots,\mathbb{A}_{{m}_{k}}$ such that $m_{1}+\cdots+m_{k}=n$
    $(k\ge 2)$.
\end{itemize}
\end{corollary}

Moreover, for any triangulation $\Gamma$ of the marked surface shown in
\Pic \ref{fig:hga tilde An}, we have $\grEnd_{\per A}\X(\tGamma^{\circlearrowleft})$ is a dg
one-cycle gentle algebra because $\tGamma^{\circlearrowleft}$ is an $\gbullet$-FFAS. Let $A$ be a hereditary algebra of type $\tA_{n}$ with arbitrary orientation.
Then by \cite[Section 7]{AG2008} and \cite[Theorem 1.1 (a)]{BZ2018}, we have the following classification of the silted algebras of $A$. Notice that the silted algebras of $A$ of the form (2) below can also founded in \cite{Y2020}.

\begin{corollary}
Let $A$ be a hereditary algebra of type $\tA_{n}$ with arbitrary orientation. Then the silted algebras of $A$ have four forms:
\begin{itemize}
\item[{\rm(1)}] the tilted algebras of type $\tA_n$;

\item[{\rm(2)}] the tilted algebras of type $\A_n$;

\item[{\rm(3)}] the direct product of some tilted algebras of type $\A_{m_1},\ldots,\A_{m_k}$ such that $m_{1}+\cdots+m_{k}=n$ $(k\ge 2)$;

\item[{\rm(4)}] the direct product of some tilted algebras of type $\A_{m_{1}},\ldots,\A_{{m}_{k}}, \tA_{\tilde{m}}$ such that $m_{1}+\cdots+m_{k} + \tilde{m}=n$ $(k\ge 1)$.
\end{itemize}
\end{corollary}

%=========================================================

\section*{acknowledgements} {The authors would like to thank the referees for many helpful comments. Houjun Zhang acknowledges support by the National Natural Science Foundation of China (Grant No. 12301051) and Natural Science Research Start-up Foundation of Recruiting Talents of Nanjing University of Posts and Telecommunications (Grant No. NY222092). Yu-Zhe Liu acknowledges support by the National Natural Science Foundation of China (Grants No. 12171207 and 12401042) and Scientific Research Foundations of Guizhou University (Grants No. [2023]16, [2022]65 and [2022]53)}.
%=========================================================
%\bibliographystyle{amsplain}
%\bibliography{class-tilt}

% 渐变色 [top color=purple, bottom color=orange]

%  \bibliographystyle{abbrv}  % 参考文献引用-abbrv格式
%  \bibliographystyle{unsrt} % 参考文献引用-unsrt格式
%  \bibliographystyle{alpha} % 参考文献引用-alpha格式
%\bibliographystyle{amsplain}
%\bibliography{referLiu20220720}

\def\cprime{$'$}
\providecommand{\bysame}{\leavevmode\hbox to3em{\hrulefill}\thinspace}
\providecommand{\MR}{\relax\ifhmode\unskip\space\fi MR }
\providecommand{\MRhref}[2]{%
  \href{http://www.ams.org/mathscinet-getitem?mr=#1}{#2}
}
\providecommand{\href}[2]{#2}

\end{document}